\documentclass[11pt]{article}
\usepackage{xcolor}
\usepackage{graphicx}
\usepackage{tikz} 
\usepackage{amsmath,amsfonts,amssymb,graphics,amsthm}
\usepackage{hyperref}
\usepackage{comment}
\usepackage{tabularx}
\usepackage[protrusion=true,expansion=true]{microtype}
\usepackage{enumerate}
\usepackage{bbm}
\usepackage{mathrsfs}
\usepackage[margin=1in]{geometry}
\usepackage[shortlabels]{enumitem}

\hypersetup{
    colorlinks=false,
    linktocpage,
    }

\numberwithin{equation}{section}

\newtheorem{theorem}{Theorem}[section]
\newtheorem{corollary}[theorem]{Corollary}
\newtheorem{lemma}[theorem]{Lemma}
\newtheorem{proposition}[theorem]{Proposition}

\newtheorem{remark}[theorem]{Remark}
\newtheorem{definition}[theorem]{Definition}

\theoremstyle{remark}

\newcommand{\bfb}{{\mathbf b}}
\newcommand{\bft}{{\mathbf t}}

\newcommand{\C}{\mathbbm{C}}
\newcommand{\D}{\mathbbm{D}}
\newcommand{\E}{\mathbbm{E}}

\newcommand{\N}{\mathbbm{N}}

\newcommand{\R}{\mathbbm{R}}
\renewcommand{\P}{\mathbbm{P}}
\newcommand{\bbH}{\mathbbm{H}}

\newcommand{\disk}{\mathrm{disk}}

\newcommand{\LF}{\mathrm{LF}}

\newcommand{\QD}{\mathrm{QD}}
\newcommand{\QS}{\mathrm{QS}}
\newcommand{\QA}{\mathrm{QA}}
\newcommand{\QP}{\mathrm{QP}}
\newcommand{\QH}{\mathrm{QH}}

\newcommand{\lexp}{{\beta}}

\renewcommand{\wp}{\eta}

\let\Re\undefined
\DeclareMathOperator{\Re}{Re}
\let\Im\undefined
\DeclareMathOperator{\Im}{Im}

\DeclareMathOperator{\CLE}{CLE}

\DeclareMathOperator{\Vol}{Vol}

\def\cX{\mathcal{X}}

\def\cM{\mathcal{M}}

\def\cI{\mathcal{I}}

\def\cF{\mathcal{F}}
\def\cE{\mathcal{E}}

\def\cC{\mathcal{C}}

\def\alb#1\ale{\begin{align*}#1\end{align*}}
\def\allb#1\alle{\begin{align}#1\end{align}}

\newcommand{\aryb}{\begin{eqnarray*}}
\newcommand{\arye}{\end{eqnarray*}}
\def\alb#1\ale{\begin{align*}#1\end{align*}}
\newcommand{\eqb}{\begin{equation}}
\newcommand{\eqe}{\end{equation}}
\newcommand{\eqbn}{\begin{equation*}}
\newcommand{\eqen}{\end{equation*}}
\newcommand{\Weld}{\mathrm{Weld}}

\newcommand{\ol}{\overline}
\newcommand{\ul}{\underline}

\newcommand{\frk}{\mathfrak}

\newcommand{\rta}{\rightarrow}

\newcommand{\wt}{\widetilde}
\newcommand{\wh}{\widehat}

\newcommand{\bdy}{\partial}

\let\originalleft\left
\let\originalright\right
\renewcommand{\left}{\mathopen{}\mathclose\bgroup\originalleft}
\renewcommand{\right}{\aftergroup\egroup\originalright}

\DeclareMathAlphabet{\mathpzc}{OT1}{pzc}{m}{it}

\begin{document}

\title{String equation for conformal loop ensemble on Liouville quantum gravity sphere}
\author{
\begin{tabular}{c}Xin Sun\end{tabular}\; 
\begin{tabular}{c}Baojun Wu \end{tabular}\; 
\begin{tabular}{c}Shengjing Xu  \end{tabular}\;
}

\date{  }

\maketitle

\begin{abstract}
Matrix models express partition functions of random surfaces with prescribed boundary lengths through string equations. We give a probabilistic realization of this structure in Liouville quantum gravity (LQG). For $\gamma^2\in(8/3,4)$, we construct the quantum $n$-hole sphere from a quantum disk decorated by an independent conformal loop ensemble (CLE) and compute its area Laplace transform at fixed boundary lengths. The transform depends only on the total boundary length apart from a simple prefactor and is determined by a hypergeometric string equation. We derive this equation from stable L\'evy excursions using LQG/CLE coupling.  The conformal welding equation yields symmetric polynomials that represent integrated LQG/CLE coupling observables. These polynomials converge to Weil--Petersson volumes of spheres with $n$ geodesic boundaries as $\gamma\downarrow0$.
\end{abstract}

\tableofcontents

\section{Introduction}

A basic problem in two-dimensional quantum gravity is to compute the partition function of a random surface with prescribed topology and boundary lengths. Matrix models, Liouville conformal field theory (Liouville CFT), and topological gravity provide different approaches to this problem. In the matrix-model approach, partition functions arise as scaling limits of generating functions for random maps. In the Liouville approach, initiated by Polyakov~\cite{polyakov-qg1}, they are expressed through a random conformal geometry. In topological gravity, the corresponding quantities are built from intersection numbers on moduli spaces of curves. A striking feature of the matrix model approach is the appearance of closely related string equations. This paper is motivated by understanding these equations directly in terms of random continuum surfaces.

Matrix models describe random surfaces through generating functions for
random maps, whose scaling limits yield partition functions with prescribed
topology and boundary lengths. For a sphere with $n\ge3$ boundary components,
Ambj{\o}rn, Jurkiewicz, and Makeenko~\cite{Ambjorn:1990ji} and Moore, Seiberg,
and Staudacher~\cite{Moore:1991ir} obtained explicit formulas for the
pure-gravity partition function. Apart from a simple prefactor,
the answer depends only on the total boundary length. Its dependence on
$n$ is generated by repeated differentiation with respect to the bulk
cosmological constant, starting from a function determined by the string
equation.
A similar structure appears in Weil--Petersson volumes of moduli spaces
of genus 0 hyperbolic surfaces with prescribed geodesic boundary lengths; see Appendix~A. This parallel
suggests a connection between the boundary amplitudes of random surfaces
and the geometry of moduli space.
The connection with topological gravity comes from the intersection-theoretic
description of Weil--Petersson volumes.
Witten conjectured that this partition function satisfies the KdV
hierarchy~\cite{Witten:1990hr}, and Kontsevich proved the
conjecture~\cite{Kontsevich1992IntersectionTO}. Mirzakhani subsequently
gave a geometric proof through her recursion for Weil--Petersson
volumes~\cite{Mirzakhani:2006eta,Mirzakhani:2006fta}.

Liouville quantum gravity (LQG) provides a continuum description of the
scaling limits of random planar maps. For $\gamma=\sqrt{8/3}$, Miller and
Sheffield identified the LQG sphere with the scaling limit of random planar maps~\cite{lqg-tbm1,lqg-tbm2,lqg-tbm3}, while Holden and Sun proved
convergence of uniform triangulations to the LQG disk under the
Cardy embedding~\cite{hs-cardy-embedding}.  These works provide the Liouville CFT description of the scaling limit of random planar maps, and Liouville CFT was recently constructed and solved by~\cite{dkrv-lqg-sphere,gkrv-bootstrap} via a probabilistic approach.
Random planar maps can also be coupled with statistical-mechanics
models, including the $O(n)$ loop model. At criticality, these
loop-decorated maps are expected to converge to LQG decorated by a
conformal loop ensemble (CLE)~\cite{msw-cle-lqg}. Introduced by
Sheffield~\cite{shef-cle}, CLEs form a one-parameter family of conformally
invariant random loop collections. In the simple-loop regime, their
laws are characterized by conformal invariance and a spatial Markov
property~\cite{shef-werner-cle}.
In this paper, we consider $\gamma$-LQG surfaces decorated by an
independent $\mathrm{CLE}_{\gamma^2}$, with $\gamma^2\in(8/3,4)$.
Using this coupling, ~\cite{Int-CLE} constructed
the quantum annulus and quantum pair of pants and computed their
fixed-boundary-length partition functions. Their work provides the
starting point for our construction with an arbitrary (finite) number of
boundary components.

In this paper, we construct the quantum $n$-hole sphere for $\gamma^2\in(8/3,4)$ and compute its area-length joint law. The answer has the derivative structure described above and is governed by a hypergeometric string equation. The main point is that this equation is derived from the excursions of a stable L\'evy process, rather than imposed as an input from a matrix model. We also prove a conformal welding identity for these surfaces and use their area transform to obtain a family of deformed Weil--Petersson polynomials. 

\subsection{The quantum $n$-hole sphere and its area distribution}
\label{subsec:intro-area}

Throughout the probabilistic construction, we take
$
\kappa=\gamma^2\in(8/3,4)$, and $
Q=\frac\gamma2+\frac2\gamma.$
In this range, the CLE loops are disjoint simple curves. To construct the quantum $n$-hole sphere, start with a quantum disk with one marked bulk point and an independent nested $\CLE_\kappa$. Let $\eta_0$ be the outermost loop surrounding the marked point. Choose an ordered collection of $n-2$ further loops from the original outermost CLE loops, using the counting measure restricted to pairwise distinct loops, all different from $\eta_0$. Removing the interiors of these $n-1$ loops leaves a surface with the topology of a sphere with $n$ boundary components. Disintegrating over their quantum lengths and applying the normalization in Definition~\ref{def: nholes} gives the measure
$\QH_n(\ell_1,\ldots,\ell_n)$ supported on the quantum surface with 
n-connected domain topology.

Here, the boundaries are labeled and $\ell_i>0$. We denote the collection $(\ell_1,\ell_2,...,\ell_n)$ by $\mathbf{\ell}$. This construction includes the quantum annulus and quantum pair of pants when $n=2,3$ which were introduced in~\cite{Int-CLE}.
The measures in this paper are generally $\sigma$-finite. For a nonnegative measurable function $F$, we write $\QH_n(\boldsymbol\ell)[F]$ for its integral. Fixed-length statements are understood in the almost-everywhere sense of Definition~\ref{def: nholes}. To state our main theorem, we recall the following ingredients for the stable L\'evy  process. Given a spectrally positive stable L\'evy process of index
$\beta=\frac4{\gamma^2}+\frac12\in(3/2,2)$,
in the time normalization for which the L\'evy measure is $x^{-1-\beta}\,dx$, let $\underline N'$ be its excursion measure above the running infimum. If $T(e)$ is the duration of an excursion and $J_e$ is its multiset of positive jump sizes, we define
\begin{equation}\label{eq:intro-excursion}
u(s;\gamma)
:=\int\left(e^{sT(e)}\prod_{x\in J_e}g(x)-1\right)\underline N'(de),
\qquad s\le0,
\end{equation}
where $g(x)=\QD(x)^\#[e^{-\mu A}]$ is the area Laplace transform under the normalized quantum-disk measure with boundary length $x$. The FZZ formula~\cite{ARS-FZZ} gives
\begin{equation}\label{eq:intro-disk-transform}
g(x)=\frac2{\Gamma(4/\gamma^2)}
\left(\frac{Mx}{2}\right)^{4/\gamma^2}K_{4/\gamma^2}(Mx),
\end{equation}
with $M=\sqrt{\frac{\mu}{\sin(\pi\gamma^2/4)}}$  and $K_\nu$ the modified Bessel function of the second kind. When the parameter $\gamma$ is clear in the context, we write $u(s)=u(s;\gamma)$.
\begin{theorem}\label{thm: 1}
    Let $A$ be the quantum area, let $\mu>0$, and set
\[
L_n=\sum_{i=1}^n\ell_i,\qquad
M=\sqrt{\frac{\mu}{\sin(\pi\gamma^2/4)}}.
\]
Then
\begin{equation}\label{thm1}
\begin{split}
\QH_n(\boldsymbol\ell)[e^{-\mu A}]
={}&\frac{\cos\!\left(\pi(4/\gamma^2-1)\right)}
{\pi  R_\gamma^{n-2}}
\frac1{\sqrt{\prod_{i=1}^n\ell_i}}\times
\left.\partial_s^{n-3}
\left[u_{}'(s)e^{L_nu_{}(s)}\right]\right|_{s=0},
\qquad n\ge3.
\end{split}
\end{equation}
Here $$R_\gamma=
\frac{(2\pi)^{4/\gamma^2-1}}
{(1-\gamma^2/4)\Gamma(1-\gamma^2/4)^{4/\gamma^2}}.$$ 
\end{theorem}
 All derivatives appearing in~\eqref{thm1} are determined explicitly by Lagrange inversion.
Formula~\eqref{thm1} shows that, apart from $\prod_i\ell_i^{-1/2}$, the formula depends only on the total boundary length $L_n$. Thus one excursion functional determines the area transforms for every number of holes. The symmetry in the boundary lengths is also visible in this formula, although the definition initially distinguishes the outer boundary and the loop surrounding the marked point.

\subsection{The hypergeometric string equation}
\label{subsec:intro-string}

A Campbell-type identity for L\'evy excursions converts~\eqref{eq:intro-excursion} into an implicit equation involving the one-dimensional integral of $g$. Evaluating this integral gives Proposition~\ref{hypergeo}. After the change of variables, 
$$u_{\rm st}(x;\gamma):=-\frac{16}{\gamma^4\pi^2}\frac{M+u(-x L_\mu(\gamma))}{2M},\qquad \text{with}\qquad L_\mu(\gamma):=\frac{2^{1-{\frac{4}{\gamma^2}}}}{\Gamma({\frac{4}{\gamma^2}})}\sqrt{2\pi}M^{{\frac{4}{\gamma^2}}+\frac{1}{2}} \frac{\pi}{\cos(\frac{4\pi}{\gamma^2})}\frac{\gamma^4\pi^2}{8}.$$
Again, when $\gamma$ is clear in the context, we write $u_{\rm st}(\lambda;\gamma)=u_{\rm st}(\lambda).$
\begin{theorem}
For $\gamma\in(\sqrt{\frac{8}{3}},2)$,  and $x\geq 0$
\begin{equation}\label{eq:intro-string}
-x=\frac{u_{\rm st}(x;\gamma)}2\,
{}_2F_1\!\left(
\frac12-\frac4{\gamma^2},\frac12+\frac4{\gamma^2};2;
-\frac{\gamma^4\pi^2}{16}u_{\rm st}(x;\gamma)
\right),
\qquad u_{\rm st}(0)=0.
\end{equation}
\end{theorem}

For every fixed $\gamma\neq0$, the analytic implicit-function theorem
shows that \eqref{eq:intro-string} determines a unique holomorphic germ
$u_{\rm st}(\,\cdot\,;\gamma)$ at $0$ satisfying
$u_{\rm st}(0;\gamma)=0$.

Although the change of variables defining $u_{\rm st}$ is singular at
the endpoint $\gamma^2=8/3$, the string equation itself has the
well-defined algebraic specialization
\[
x+\frac12u_{\rm st}(x)
+\frac{2\pi^2}{9}u_{\rm st}(x)^2=0.
\]
This is the genus-zero pure-gravity string equation familiar from the
matrix-model literature~\cite{Ambjorn:1990ji,Moore:1991ir}. 
\subsection{Conformal welding and deformed Weil--Petersson volumes}
\label{subsec:intro-wp}

The quantum $n$-hole sphere can be placed back into a quantum sphere by welding a one-point quantum disk to each of its boundary component. See Theorem~\ref{thm: welding QSn}.

Let $\Gamma$ be sampled from the full-plane nested
$\CLE_\kappa^{\C}$.  Fix $n$ distinct points 
$\boldsymbol z=(z_1,\ldots,z_n)$ in $\C$.
For each $i$, let $\eta_i=\eta_i(\Gamma,\boldsymbol z)$ be the
almost surely unique outermost loop of $\Gamma$ that surrounds $z_i$
and none of the other marked points.  Write
$
\boldsymbol\eta=(\eta_1,\ldots,\eta_n),$
and let
$
D_{\boldsymbol\eta}$ be the multiply connected component of
$
\widehat{\C}\setminus\bigcup_{i=1}^n\eta_i$.
Let $\mathcal C_n(\Gamma,\boldsymbol z)$ be the event that every loop
of $\Gamma$ contained in $D_{\boldsymbol\eta}$ is contractible in
$D_{\boldsymbol\eta}$.  We define the finite kernel
$\mathsf m_n(d\boldsymbol\eta\mid\boldsymbol z)$ by
\begin{align}\label{def:mn}
\int F(\boldsymbol\eta)\,
\mathsf m_n(d\boldsymbol\eta\mid\boldsymbol z)
:=
\mathbb E_{\CLE_\kappa^{\C}}\left[
F\bigl(\boldsymbol\eta(\Gamma,\boldsymbol z)\bigr)
\mathbf 1_{\mathcal C_n(\Gamma,\boldsymbol z)}
\right]
\end{align}
for every nonnegative measurable function $F$.
For insertion weights $\alpha_1,\ldots,\alpha_n$, put
\[
\Delta_\alpha=\frac\alpha2\left(Q-\frac\alpha2\right),\qquad
\mathsf m_n^{(\alpha_i)_i}(d\boldsymbol\eta\mid\boldsymbol z)
=\prod_{i=1}^n
\left(\frac{\mathrm{CR}(\eta_i,z_i)}2\right)^{2\Delta_{\alpha_i}-2}
\mathsf m_n(d\boldsymbol\eta\mid\boldsymbol z),
\]
where $\mathrm{CR}(\eta_i,z_i)$ is the conformal radius at $z_i$ of the disk bounded by $\eta_i$. The quantity of interest is the Liouville correlation function weighted by the total mass of this loop kernel and integrated over the remaining marked points.

Fix $z_1=0$, $z_2=1$, and $z_3=e^{i\pi/3}$. Let $\alpha_i\in (Q-\frac{\gamma}{4},Q)$ with $1\leq i\leq n$ and $n\geq 4$.
For $\boldsymbol{\lambda}=(\lambda_1,\lambda_2,..,\lambda_n):=(-2i\frac{Q-\alpha_1}{\gamma},..,-2i\frac{Q-\alpha_n}{\gamma})$, we define the $\gamma$-deformed Weil-Petersson volume $V_{0,n}^{\rm WP}(\boldsymbol{\lambda})$ as
    \begin{align}
   \int_{\mathcal{U}_n} \LF_\C^{{(\alpha_i, z_i)}_{i}}[e^{-\mu A}] \times \left|\mathsf m_n^{(\alpha_i)_i}(\cdot\mid\boldsymbol z)\right| \prod_{i=4}^{n} d^2z_i:=C_{\rm WP}(\gamma,\boldsymbol\alpha,\mu,n) \prod_i \frac{1}{\cosh\pi \lambda_i} V_{0,n}^{\rm WP}(\boldsymbol{\lambda})
    \end{align}
The normalization $C_{\rm WP}$ is explicit in~\eqref{cwp} and 
$$\mathcal{U}_n:=\{(z_4,..,z_n)\in \hat{\C}^{n-3}\mid z_i\neq z_j \text{ for }4\leq i < j\leq n, z_k\neq z_1,z_2,z_3 \text{ for } 4\leq k\leq n\}$$
\begin{theorem}
\begin{align}\label{eq:legendre wp}
    V_{0,n}^{\rm WP}(\boldsymbol{\lambda})=\lim _{x \rightarrow 0}-\frac{1}{2}\left(\frac{\partial}{\partial x}\right)^{n-3} u_{\mathrm{\rm st}}^{(1)}(x) \prod_{i=1}^n P_{-\frac{1}{2}-i \lambda_i}\left(1+ \frac{\gamma^4\pi^2}{8} u_{\mathrm{st}}(x)\right) 
 \end{align} 
 Here $P_\nu$ denotes the Legendre function of the first kind, with the
branch analytic at $1$ and normalized by $P_\nu(1)=1$.
\begin{align}\label{eq:wplimit}
    \lim_{\gamma\to 0}V_{0,n}^{\rm WP}(\frac{2}{\pi\gamma^2}b_1,\frac{2}{\pi\gamma^2}b_2,...,\frac{2}{\pi\gamma^2}b_n)=\Vol_{\rm WP}(\mathcal{M}_{0,n}(b_1,..,b_n))
\end{align}
      Here $\Vol_{\rm WP}(\mathcal{M}_{0,n}(b_1,..,b_n))$ is the Weil-Petersson volume of n-hole sphere with boundary length $b_i$.
\end{theorem}
 Since only a finite Taylor expansion contributes, the right-hand side of Equation~\eqref{eq:legendre wp} is a symmetric polynomial in $\lambda_1^2,\ldots,\lambda_n^2$ of total degree at most $n-3$. 
For example,
\[
V_{0,4}^{\rm WP}\!\left(
\frac{2b_1}{\pi\gamma^2},\ldots,\frac{2b_4}{\pi\gamma^2}
\right)
=2\pi^2+\frac{3\pi^2\gamma^4}{32}
+\frac12\sum_{i=1}^4b_i^2.
\]
The terminology deformed Weil--Petersson volume is motivated by its $\gamma\downarrow 0$ limit~\eqref{eq:wplimit}.

\subsection{Outlook and perspectives}

\paragraph{1. Ising correlations coupled to LQG.}
A natural next step is to extend our construction and area formula to
$\mathrm{CLE}_\kappa$ with $\kappa\in(4,8)$ and $\gamma=4/\sqrt\kappa$.
The generalized quantum annuli and pairs of pants of~\cite{Int-CLE}
provide a starting point. 
At $\kappa=16/3$, suitable combinations of multipoint gasket Green
functions should recover Ising spin correlations through the FK
representation. Integrating these against Liouville correlations at
$\gamma=\sqrt3$ would yield Ising-gravity correlation numbers, providing
a probabilistic approach to the correlation-number problem studied by
Belavin and Zamolodchikov~\cite{Belavin:2006ex}.

\paragraph{2. A probabilistic construction for small $\gamma$.}
Can one construct a probabilistic counterpart of the quantum $n$-hole
sphere for small $\gamma$? Our CLE construction is restricted to
$\gamma^2>8/3$, below which no nontrivial $\mathrm{CLE}_{\gamma^2}$
exists~\cite{shef-werner-cle}. An alternative construction could give
a geometric interpretation of the Weil--Petersson limit as $\gamma\to0$.

\paragraph{3. Conformal welding and topological recursion.}
Can conformal welding yield topological recursion for the
fixed-boundary-length partition functions of quantum surfaces?
A recursive cutting and welding construction, including in positive
genus, could provide a probabilistic interpretation of matrix-model
loop equations and recover Mirzakhani's recursion in the
$\gamma\to0$ limit.

\paragraph{Organization of the paper.}
Section~\ref{sec:prelim} fixes our conventions for quantum surfaces, Liouville fields, CLE, and conformal welding. Section~\ref{sec: quantum n hole} constructs the quantum $n$-hole sphere, proves the welding identities, and reduces its area transform to L\'evy excursion functionals. Section~\ref{string equation} derives the string equation and computes the required derivatives. Section~5 treats the disk-welding transform and the deformed Weil--Petersson polynomials, including their conditional LQG/CLE interpretation and their algebraic limit.  Appendix~A recalls the genus-zero intersection-number and Weil--Petersson formulas, and Appendix~B proves the Bessel integral identity used in the computation.

\paragraph{Acknowledgement:} We thank Zhuo Wu for helpful discussion on Lemma 3.15.  X.S. and B.W. were supported by National Key R$\&$D Program
of China (No. 2023YFA1010700). X.S. was also partially supported by the NSF grant DMS-2027986,
the NSF Career grant DMS-2046514, and a start-up grant from the University of Pennsylvania. S.X. was partially supported by NSF Grant
DMS-2331096.

\section{Preliminaries}\label{sec:prelim}
{We assume familiarity with CLE in the simple-loop regime; see~\cite{shef-werner-cle,shef-cle}. This section fixes the measure-theoretic, LQG, LCFT, and CLE conventions used below.}

\subsection{Measure theoretic background}\label{subsec:measure-pre}
{We use probabilistic terminology for sigma-finite measures. Let $M$ be a sigma-finite measure on a standard Borel space $(\Omega,\cF)$ and let $X:(\Omega,\cF)\to(E,\cE)$ be measurable. For every nonnegative measurable function $f$, write}
\[
M[f]:=\int_\Omega f\,dM.
\]
We call the pushforward measure $M_X:=X_*M$ on $(E,\cE)$ the law of $X$; thus $M_X[f]=M[f(X)]$.
If $0<|M|<\infty$, then $M^\#:=|M|^{-1}M$ denotes the corresponding probability measure. The next Lemma is from~\cite[Lemma 2.2]{Int-CLE}\label{def:disint}.

\begin{lemma}[Disintegration]
Assume that $X:\Omega\to\R^n$ is measurable and that $M_X$ is absolutely continuous with respect to Lebesgue measure. A measurable family $\{M_x:x\in\R^n\}$ of measures on $(\Omega,\cF)$ is a \emph{disintegration of $M$ over $X$} if, for every $A\in\cF$ and every nonnegative Borel function $f$,

\begin{equation}\label{eq:disint}
\int_A f(X)\,dM=\int_{\R^n}f(x)M_x(A)\,d^nx.
\end{equation}
Such a disintegration exists and is unique for Lebesgue-almost-every $x$.
\end{lemma}

\subsection{Liouville quantum gravity and Liouville conformal field  theory}\label{subsec:Liouville-CFT}
In this section, we review the precise definitions of some $\gamma$-LQG surfaces and Liouville fields.
For more background, we refer to~\cite{ghs-mating-survey,vargas-dozz-notes} and references  therein, as well as the preliminary sections in \cite{AHS-SLE-integrability, ARS-FZZ}.
 
\subsubsection{Gaussian free field, Liouville field, and the DOZZ formula}\label{subsubsec:GFF}

Let $\cX$ be either $\C$ or the upper half-plane $\bbH$, equipped with
a smooth conformal metric whose metric completion is compact.  Let
$H^1(\cX)$ be the corresponding first-order Sobolev space and let
$H^{-1}(\cX)$ be its continuous dual.  Put $|z|_+:=\max\{|z|,1\}$ and
\begin{align*}
G_\bbH(z,w)
&=-\log|z-w|-\log|z-\overline w|
  +2\log|z|_++2\log|w|_+,\qquad z,w\in\bbH,\\
G_\C(z,w)
&=-\log|z-w|+\log|z|_++\log|w|_+,\qquad z,w\in\C.
\end{align*}
Let $h_\cX$ be the centered Gaussian random distribution with covariance
kernel $G_\cX$; equivalently, for smooth test functions $f_1,f_2$,
\[
\E[(h_\cX,f_1)(h_\cX,f_2)]
=\iint_{\cX\times\cX}f_1(z)G_\cX(z,w)f_2(w)\,d^2z\,d^2w.
\]
Thus $h_\C$ is the whole-plane GFF and $h_\bbH$ is the free-boundary
GFF on $\bbH$, in each case normalized to have mean zero on the unit
semicircle or circle, as appropriate.  We denote its law by $P_\cX$.

We now review the Liouville fields on $ \C$  and $\bbH$ following~\cite[Section 2.2]{AHS-SLE-integrability}.
\begin{definition}
	\label{def-LF-sphere}
	Suppose $(h, \mathbf c)$ is sampled from $P_\C \times [e^{-2Qc}dc]$ and set $\phi(z):=h(z)-2Q\log|z|_++\mathbf c$. 
	Then we write  $\LF_{\C}$ as the law of $\phi$ and call a sample from  $\LF_{\C}$  a Liouville field on $\C$.
	
	Suppose $(h, \mathbf c)$ is sampled from $P_\bbH \times [e^{-Qc}dc]$ and set $\phi(z):=h(z)-2Q\log|z|_++\mathbf c$. 
	Then we write  $\LF_{\bbH}$ as the law of $\phi$ and call a sample from  $\LF_{\bbH}$  a \emph{Liouville field on $\bbH$}.
\end{definition}
We also need the following Liouville fields  with   insertions.  
\begin{definition}\label{def-RV-sph}
	Let $(\alpha_i,z_i) \in  \R \times \C$ for $i = 1, \dots, m$, where $m \ge 1$ and the $z_i$'s are distinct. 
	Let $(h, \mathbf c)$ be sampled from $ C_\C^{(\alpha_i,z_i)_i}  P_\C \times [e^{(\sum_i \alpha_i  - 2Q)c}dc]$ where
	\[C_{  \C}^{(\alpha_i,z_i)_i}=\prod_{i=1}^m |z_i|_+^{-\alpha_i(2Q -\alpha_i)} e^{\sum_{i < j} \alpha_i \alpha_j G_\C(z_i, z_j)}.\]
	Let \(\phi(z) = h(z) -2Q \log |z|_+  + \sum_{i=1}^m \alpha_i G_\C(z, z_i) + \mathbf c\).
	We write  $\LF_{ \C}^{(\alpha_i,z_i)_i}$ for the law of $\phi$ and call a sample from  $\LF_{ \C}^{(\alpha_i,z_i)_i}$ 
	a \emph{Liouville field on $ \C$ with insertions $(\alpha_i,z_i)_{1\le i\le m}$}. 
\end{definition}
\begin{definition}\label{def-1pt-H}
	For $\alpha \in \R$ and $z_0 \in \bbH$, let $(h, \mathbf c)$ be sampled from $(2\Im z_0)^{-\alpha^2/2} |z_0|_+^{-2\alpha (Q-\alpha)}P_\bbH\times [e^{(\alpha-Q)c}dc]$. Let $\phi(z) = h(z) - 2Q \log |z|_+ + \alpha G_\bbH(z, z_0) + \mathbf c$. We write $\LF_\bbH^{(\alpha, z_0)}$ for the law of $\phi$ and call a sample from $\LF_\bbH^{(\alpha, z_0)}$ a \emph{Liouville field on $\bbH$ with insertion $(\alpha, z_0)$}. 
\end{definition}

Fix $\gamma\in(0,2)$ and let $h$ have law $P_\cX$, where
$\cX\in\{\C,\bbH\}$.  Write $h_\varepsilon(z)$ for the usual
circle average in the interior and semicircle average on $\partial\bbH$.
The regularized measures
\[
\mu_h^\varepsilon(d^2z)
:=\varepsilon^{\gamma^2/2}e^{\gamma h_\varepsilon(z)}\,d^2z
\]
converge in the local weak topology to the quantum area measure $\mu_h$.
For $\cX=\bbH$, the boundary regularizations
\[
\nu_h^\varepsilon(dx)
:=\varepsilon^{\gamma^2/4}e^{\gamma h_\varepsilon(x)/2}\,dx
\]
converge to the quantum boundary-length measure $\nu_h$; see
\cite{shef-kpz,shef-wang-lqg-coord}.  The same GMC construction is used
for the Liouville fields above, including insertions whenever the relevant
local integrability conditions hold.

\subsubsection{Quantum  sphere and Quantum disk}\label{subsub:quantum-surface}

A \emph{quantum surface} is an equivalence class of pairs $(D, h)$ where {$D\subset\widehat\C$ is a domain} and $h$ is a generalized function on $D$.  
For $\gamma\in (0,2)$, we say that 
$(D, h) \sim_\gamma (\wt D, \wt h)$ if there is a conformal map $\psi: \wt D \to D$ such that 
\eqb\label{eq-QS}
\wt h = h \circ \psi + Q \log |\psi'|. 
\eqe
We write $ (D, h)/{\sim_\gamma}$ as the quantum surface corresponding to $(D,h)$. 
An \emph{embedding} of a quantum surface is a choice of its representative.
Both the notions of quantum area and quantum length measures are  intrinsic to the quantum surface~\cite{shef-kpz,shef-wang-lqg-coord}.

{We equip the genus-zero quantum-surface spaces used below with
the standard Borel structures obtained from normalized embeddings, as
in~\cite[Section 2.3]{Int-CLE}; consequently, the disintegrations in
Lemma~\ref{def:disint} apply to these quotient spaces.}

We can also consider quantum surfaces decorated with other structures. For example, 
let {$n\in\mathbb Z_{\ge0}$} and $\cI$ be an at most countable index set, consider tuples $(D, h, (\eta_i)_{i\in \cI}, z_1,\cdots,z_n)$ such that $D$ is a domain, 
$h$ is a distribution on $D$, $\eta_i$ are {unparameterized loops (continuous maps from $S^1$ modulo reparameterization) in $\overline D$} and
$z_i \in D\cup \bdy D$. We say that
{
\[
(D,h,(\eta_i)_{i\in\cI},z_1,\ldots,z_n)
\sim_\gamma
(\wt D,\wt h,(\wt\eta_i)_{i\in\cI},\wt z_1,\ldots,\wt z_n).
\]
}
if there is a conformal map $\psi: \wt D \to D$ such that~\eqref{eq-QS} holds, $\psi(\wt z_i) = z_i$ for all $1\le i\le n$, and  $\psi\circ \wt \eta_i=\eta_i$ for all $i\in \cI$.
We call an equivalence class defined through ${\sim_\gamma}$ a \emph{decorated quantum surface}, and  call  a choice of its representative an embedding.

The quantum sphere is a class of quantum surfaces with the sphere topology introduced in~\cite{wedges}. Let $\QS_n$ be the law of the quantum sphere with $n$ marked points, and $\QS = \QS_0$. We will need $\QS, \QS_2$ and $\QS_3$.
Originally, $\QS_2$ was  defined first as in \cite[Section 4.5]{wedges}, and $\QS$ and $\QS_2$ were defined by adding or removing marked points from $\QS_2$. It is convenient to use the following definition. The equivalence with the original definition is explained in \cite[Remark 2.30]{AHS-SLE-integrability}.
\begin{definition}\label{def-QS-2}
Let $(u_1, u_2, u_3) = (0, 1, e^{\mathbf{i}\pi/3})$,  and sample $\phi$ from $\frac{\pi \gamma}{2(Q-\gamma)^2}\LF_\C^{(\gamma, u_1),(\gamma, u_2),(\gamma, u_3)}$. 
Let $\QS_3$ be the law of the decorated quantum surface $(\wh\C, \phi,u_1,u_2,u_3)/{\sim_\gamma}$. We call a sample from $\QS_3$ a \emph{quantum sphere with three marked points}. 
\end{definition}
\begin{definition}\label{def-QS}
For $(\wh\C,h,u_1,u_2,u_3)/{\sim_\gamma}$ sampled from $\mathcal{A}_h(\C)^{-1} \QS_3$,
let $\QS_2$ be the law of the decorated quantum surface $(\wh\C, h,u_1,u_2)/{\sim_\gamma}$. 	
For $(\wh\C,h,u_1,u_2,u_3)/{\sim_\gamma}$ sampled from $\mathcal{A}_h(\C)^{-3}\QS_3$, let $\QS$ be the law of the quantum surface $(\wh\C,h)/{\sim_\gamma}$. For {$n\geq 4$}, let $(\cC,h)$ be a sample from $\mu_h(\cC)^n
\QS$, and then independently sample $z_i$ ($1\leq i\leq n$)  from the probability measure proportional to $\mu_h$.
 Then we let $\QS_n$ be 
 the law of $(\cC,h,z_1,z_2,{\ldots},z_n)/{\sim_\gamma}$.
\end{definition}

The quantum disk is the most canonical quantum surface with the disk topology. Let $\QD_{m,n}$ be the law of the quantum disk with $m$ marked bulk points and $n$ marked boundary points. We will need $\QD$, $\QD_{1,0}$, $\QD_{1,1}$ $\QD_{0,2}$, and $\QD_{0,3}$.
Historically, $\QD_{0,2}$ was defined first in \cite{wedges}, and other $\QD_{m,n}$ were defined in terms of $\QD_{0,2}$ by adding and removing points. For our purpose it is convenient use
the equivalent definition of $\QD_{1,0}$ using the Liouville field \cite[Theorem 3.4]{ARS-FZZ}, then define $\QD$, $\QD_{1,1}$ and $\QD_{0,3}$ in terms of $\QD_{1,0}$.

Let $\{\LF_\bbH^{(\alpha, \mathbf{i})}(\ell): \ell >0 \}$ be the disintegration of $\LF_\bbH^{(\alpha,\mathbf{i})}$  over the quantum boundary length. That is, each measure $\LF_\bbH^{(\alpha, \mathbf{i})}(\ell)$ is supported on the set of fields having quantum boundary length $\ell$, and $\LF_\bbH^{(\alpha, \mathbf{i})} = \int_0^\infty \LF_\bbH^{(\alpha, \mathbf{i})}(\ell)\, d\ell$. 
See \cite[Lemma 4.3]{ARS-FZZ} for an explicit construction of $\LF_\bbH^{(\alpha, \mathbf{i})}(\ell)$. Since the quantum boundary length is intrinsic to quantum surfaces, we can now define a quantum surface with quantum boundary length $\ell$ and having a log singularity with coefficient $\alpha$.

\begin{definition}\label{def-QD-alpha}
For $\ell>0$,
we let $\cM_{1,0}^\disk(\alpha;\ell )$ be the law of the quantum surface $(\bbH, \phi, \mathbf{i})/{\sim_\gamma}$ 
with $\phi$ sampled from $\LF_\bbH^{(\alpha, \mathbf{i})} (\ell)$. Define $\QD_{1,0}(\ell) := \frac\gamma{2\pi (Q-\gamma)^2}\cM_{1,0}^\disk(\gamma;\ell )$.
\end{definition}
\begin{definition}\label{def-QD}\label{def-QD0203}
Let $\ell>0$. Sample $(D,h,z)/{\sim_\gamma}$ from the weighted measure  $\mathcal{A}_h(D)^{-1}\QD_{1,0}(\ell)$, and let $\QD$ be the law of $(D,h)/{\sim_\gamma}$. For $(D,h,z)/{\sim_\gamma}$ sampled from $\ell\QD_{1,0}(\ell)$, sample $p\in \bdy D$ from the probability measure proportional to quantum boundary length measure. Let $\QD_{1,1}(\ell)$ be the law of $(D,h,z, p)/{\sim_\gamma}$. Let $n\in\N$. For $(D,h)/{\sim_\gamma}$ sampled from $\ell^n\QD(\ell)$, independently sample $(p_i)_{1\le i\le n}\in\partial D$ from the probability measure proportional to quantum boundary measure. Let $\QD_{0,n}(\ell)$ be the law of $(D,h,(p_i)_{1\le i\le n})/{\sim_\gamma}$.
\end{definition}

One can fix the embedding of $\cM^{\disk}_{1,0}(\alpha;\ell)$ as follows:
\begin{proposition}[{\cite[Proposition 2.21]{Int-CLE}}]\label{lem:har}
	For $\alpha > \frac\gamma2$ and $\ell > 0$, let $(D,h,z)$ be an embedding of a sample from $\cM^{\disk}_{1,0}(\alpha;\ell)$. Given $(D,h,z)$, let $p$ be a point sampled from the harmonic measure on $\bdy D$ viewed from $z$, then the law of  $(D,h,z,p)/{\sim_\gamma}$ equals that of $(\bbH, X, \mathbf{i},0)/{\sim_\gamma}$ where $X$ is sampled from $\LF_{\bbH}^{(\alpha,{\mathbf{i}})}(\ell)$.
\end{proposition}

\subsubsection{{Some integrability results in LCFT}}\label{subsub:integrable}
We now review the DOZZ formula for the structure constant of LCFT.
\begin{theorem}[{\cite{krv-dozz}}]\label{prop-DOZZ}
	Suppose $\alpha_1, \alpha_2, \alpha_3$ satisfy the \emph{Seiberg bounds}
	\eqb\label{eq-seiberg}
	\sum_{i=1}^3 \alpha_i > 2Q, \qquad \textrm{and}\qquad  \alpha_i < Q\text{ for }i=1,2,3.
	\eqe
	Let $(z_1, z_2, z_3) = (0, 1, {e^{\mathbf i\pi/3}})$. Then with $\mu>0$ and $C^\mathrm{DOZZ}_\gamma(\alpha_1,\alpha_2,\alpha_3)$ defined in \cite[Formula 1.12]{krv-dozz}, we have
	\[\LF_\C^{(\alpha_i, z_i)_i}[ e^{-\mu \mu_\phi(\C)}] = \frac12 C^\mathrm{DOZZ}_\gamma(\alpha_1,\alpha_2,\alpha_3)
	\mu^{\frac{2Q-\alpha_1-\alpha_2-\alpha_3}{\gamma}}. \]
\end{theorem}

We first recall the law of the quantum boundary length  under $\LF_\bbH^{(\alpha, i)}$ obtained in~\cite{remy-fb-formula},  following the presentation of~\cite[Proposition 2.8]{ARS-FZZ}.
\begin{proposition}[{\cite{remy-fb-formula}}]\label{prop-remy-U}
For $\alpha > \frac{\gamma}{2}$,  the law of the quantum length $\nu_\phi(\R)$ under $\LF_\bbH^{(\alpha, i)}$ is 
{
\[
\mathbf1_{\ell>0}\frac2\gamma 2^{-\alpha^2/2}\ol U(\alpha)
\ell^{\frac2\gamma(\alpha-Q)-1}\,d\ell,
\]
}
where
	\eqb\label{eq:U0-explicit}
	\ol U(\alpha) = \left( \frac{2^{-\frac{\gamma\alpha}2} 2\pi}{\Gamma(1-\frac{\gamma^2}4)} \right)^{\frac2\gamma(Q-\alpha)} 
	\Gamma( \frac{\gamma\alpha}2-\frac{\gamma^2}4).
	\eqe
\end{proposition}
Let $\{\LF_\bbH^{(\alpha, i)}(\ell): {\ell>0} \}$ be the disintegration of $\LF_\bbH^{(\alpha, i)}$  over $\nu_\phi(\R)$.
Namely, for each {pair of nonnegative measurable functions $f$ on $(0,\infty)$ and $g$ on $H^{-1}(\bbH)$},
\begin{equation}\label{eq:field-dis}
\LF_\bbH^{(\alpha, i)} [f(\nu_\phi(\R))g(\phi) ]  =   \int_0^\infty f(\ell) \LF_\bbH^{(\alpha, i)}(\ell) [g(\phi)] \, d\ell.
\end{equation} 
Although the general theory of disintegration  only defines $ \LF_\bbH^{(\alpha, i)}(\ell)$ for almost every $\ell\in (0,\infty)$, 
the following lemma describes a canonical version of $\LF_\bbH^{(\alpha, i)} (\ell)$ for every $\ell>0$.
\begin{lemma}[{\cite[Lemma 4.3]{ARS-FZZ}}]\label{lem:field-disk}
		Let $h$ be a sample from $P_\bbH$ and $\hat h(\cdot)=  h(\cdot) -2Q \log \left|\cdot\right|_+ +\alpha G_\bbH(\cdot,i)$.  
		Then the law of $\hat h+\frac{2}{\gamma}\log \frac\ell{\nu_{{\hat h}}(\R)}$ under the reweighted measure $2^{-\alpha^2/2} \frac2\gamma \ell^{-1} \left(\frac\ell{\nu_{{\hat h}}(\R)}\right)^{\frac2\gamma(\alpha-Q)}  P_\bbH$ is a version of the disintegration $\{\LF_\bbH^{(\alpha, i)}(\ell): \ell >0 \}$.
\end{lemma}

We now recall the quantum disk with one generic bulk insertion.
\begin{definition}[{\cite{ARS-FZZ}}]\label{def-QD-alpha}
For 
$\ell>0$,
we let $\cM_1^\disk(\alpha;\ell )$ be the law of $(\bbH, \phi, i)/{\sim_\gamma}$ 
where $\phi$ is sampled from $\LF_\bbH^{(\alpha, i)} (\ell)$.
\end{definition}

The FZZ formula is the {analogue} of the DOZZ formula for  $\LF^{ (\alpha,i)}_\bbH$~proposed in~\cite{FZZ}  and proved in~\cite{ARS-FZZ}.	
We record the most convenient form for our purpose, which uses the modified Bessel function of the second kind $K_\nu(x)$ \cite[Section 10.25]{NIST:DLMF}. 
One concrete representation of $K_\nu(x)$ {in the range of interest} is the following \cite[(10.32.9)]{NIST:DLMF}:
\eqb\label{eq-Kv}
K_\nu(x) := \int_0^\infty e^{-x \cosh t} \cosh(\nu t) \, dt \quad \text{ for } x > 0 \text{ and } \nu \in \R.
\eqe
\begin{theorem}[{\cite[Theorem 1.2, Proposition 4.20]{ARS-FZZ}}]\label{thm-FZZ}
	For $\alpha \in (\frac\gamma2, Q)$ and $\ell>0$,  let $A$ be the quantum area of a sample from $\cM_1^\disk(\alpha; \ell)$. 
	{Recall $\ol U(\alpha)$ from
	Proposition~\ref{prop-remy-U}.  For $\mu>0$,}
	\begin{align}\label{eq:malpha}
	\cM_1^\disk(\alpha;\ell)[e^{-\mu A}]
	={}\frac2\gamma 2^{-\alpha^2/2}\ol U(\alpha)\ell^{-1}
	\frac2{\Gamma(\frac2\gamma(Q-\alpha))}\times\left(\frac12
\sqrt{\frac\mu{\sin(\pi\gamma^2/4)}}\right)^{\frac2\gamma(Q-\alpha)}
	K_{\frac2\gamma(Q-\alpha)}\!\left(
	\ell\sqrt{\frac\mu{\sin(\pi\gamma^2/4)}}\right).
	\end{align}
In particular,
\begin{align}
\QD(\ell)[e^{-\mu A}]
&=|\QD(\ell)|\,\QD(\ell)^\#[e^{-\mu A}],\nonumber\\
\QD(\ell)^\#[e^{-\mu A}]
&=\frac2{\Gamma(4/\gamma^2)}
\left(\frac\ell2\sqrt{\frac\mu{\sin(\pi\gamma^2/4)}}\right)^{4/\gamma^2}
K_{4/\gamma^2}\!\left(
\ell\sqrt{\frac\mu{\sin(\pi\gamma^2/4)}}\right),\label{eq:unitarea}\\
|\QD(\ell)|&=R_\gamma\ell^{-2-4/\gamma^2},
\qquad
R_\gamma:=
\frac{(2\pi)^{4/\gamma^2-1}}
{(1-\gamma^2/4)\Gamma(1-\gamma^2/4)^{4/\gamma^2}}.\label{eq:Rgamma}
\end{align}
\end{theorem}

\subsection{Random loops and conformal welding}

Let $(S_i,B_i,\nu_i)$, $i=1,2$, be oriented bordered Riemann surfaces
with distinguished boundary components carrying finite measures with the
same positive total mass.  A conformal welding identifies $B_1$ and $B_2$ by their
boundary-length coordinates and produces an oriented surface $S$ with a
distinguished interface $\eta$.  Equivalently, the two components of
$S\setminus\eta$ are conformally equivalent to $S_1$ and $S_2$, and the
two pushforward boundary measures agree on $\eta$.  When boundary roots are
needed to remove rotational ambiguity, we sample them independently from
$\nu_i/|\nu_i|$ before welding; this is the \emph{uniform welding}.

For $\kappa\in(8/3,4)$, the loops of $\CLE_\kappa$ are disjoint Jordan
curves~\cite{shef-werner-cle}.  We use $\CLE_\kappa^\D$ for nested
CLE in a simply connected domain and $\CLE_\kappa^\C$ for its full-plane,
M\"obius-invariant version~\cite{Kemppainen2014TheNS}.  If two independent
samples from $\QD(\ell)$ are given independent quantum-typical boundary
roots, their conformal welding is almost surely unique modulo conformal
automorphisms~\cite{shef-zipper}.  We write
$\Weld(\QD(\ell),\QD(\ell))$ for the resulting loop-decorated law.

We now recall the sphere decomposition obtained from
\cite[Theorem~7.1 and Proposition~9.3]{ACSW24a}.
Let $\CLE_\kappa^\C(d\Gamma)$ be full-plane CLE, let
$\mathrm{Count}_\Gamma(d\eta)$ be counting measure on its loops, and let
$\mathbb F$ be an embedding measure whose pushforward to quantum
surfaces is $\QS_2$.  Write $E$ for the event
that $\eta$ separates $0$ from $\infty$.
\begin{proposition}\label{disk-welding}
{The pushforward of
$\mathbf1_E\mathbb F(dh)\mathrm{Count}_\Gamma(d\eta)
\CLE_\kappa^\C(d\Gamma)$ to loop-decorated quantum surfaces is}
\begin{align*}
C(\gamma)\int_0^\infty \ell\,
\Weld\!\left(
\QD_{1,0}(\ell)\otimes\CLE_\kappa^\D,
\QD_{1,0}(\ell)\otimes\CLE_\kappa^\D
\right)d\ell,
\end{align*}
where
\[
C(\gamma)=\frac1{4\pi}
\frac{\Gamma(\gamma^2/4)\Gamma(1-\gamma^2/4)}{(Q-\gamma)^2}
\tan\!\left(\pi\left(\frac4{\gamma^2}-1\right)\right).
\]
The welding is uniform with respect to the quantum boundary-length
measures and retains the welded loop, the marked points corresponding
to $0$ and $\infty$ on their respective sides, and the two CLE
decorations.
\end{proposition}

If $(D,h)/{\sim_\gamma}$ is sampled from $\QD$ and an independent
$\CLE_\kappa^\D$ is drawn in an embedding, conformal invariance makes the law
of $(D,h,\Gamma)/{\sim_\gamma}$ independent of that embedding.  We denote
the resulting measure by $\QD\otimes\CLE_\kappa$.  The notation
$\QD(a)^\#\otimes\CLE_\kappa$, $\QD_{1,0}\otimes\CLE_\kappa$, and its
fixed-length variants has the analogous meaning.

Fix $a>0$ and embed a sample from
$\QD(a)^\#\otimes\CLE_\kappa$.  For $\wp\in\Gamma$, let $D_\wp$ be the
component of $D\setminus\wp$ which does not contain $\partial D$.  The loop
$\wp$ is \emph{outermost} if $D_\wp$ is not contained in $D_{\eta'}$ for
any other $\eta'\in\Gamma$.  List the quantum lengths of the outermost loops
in non-increasing order as $(\ell_i)_{i\ge1}$.  The next two propositions describe the law of the outermost-loop
lengths and, conditionally on these lengths, the laws of the
encircled quantum surfaces.

\begin{proposition}[\cite{msw-cle-lqg,bbck-growth-frag,ccm-perimeter-cascade}]\label{prop-ccm}
{Set $\beta:=\frac4\kappa+\frac12\in(\frac32,2)$. Let $(\zeta'_t)_{t\ge0}$ be a spectrally positive $\beta$-stable L\'evy process whose L\'evy measure is $\mathbf1_{\{x>0\}}x^{-\beta-1}\,dx$, and denote its law by $\P^\beta$.  Put
\[
\tau_{-a}:=\inf\{t\ge0:\zeta'_t=-a\},
\]
and let $(x_i)_{i\ge1}$ be the non-increasing sequence of jump sizes of
$\zeta'$ on $[0,\tau_{-a}]$.  Then $(\ell_i)_{i\ge1}$ has the law of
$(x_i)_{i\ge1}$ under the probability measure
\[
\frac{\tau_{-a}^{-1}}{\E^\beta[\tau_{-a}^{-1}]}\,\P^\beta.
\]}
\end{proposition}
\begin{proposition}[\cite{msw-cle-lqg}]\label{prop-msw}
{Conditionally on $(\ell_i)_{i\ge1}$, the quantum surfaces
$(D_\wp,h|_{D_\wp})/{\sim_\gamma}$ encircled by the outermost loops are
independent, and the surface whose boundary length is $\ell_i$ has law
$\QD(\ell_i)^\#$.}
\end{proposition}


\section{Quantum $n$-hole {surfaces}}\label{sec: quantum n hole}

In this section, we introduce the quantum $n$-hole surface. In Section~\ref{subsec:def-QHn}, we first present its definition and give conformal welding results. Then, in Section~\ref{subsec:partition}, we compute the area distribution of {a quantum $n$-hole surface}.

\subsection{Definitions and the statements of conformal welding results}\label{subsec:def-QHn}

Let us start with the definition. We assume $\kappa\in (\frac83,4)$ and $\gamma=\sqrt{\kappa}$ throughout this section. The probability measure for the disk ({resp. sphere}) nested $\CLE_\kappa$ is denoted by $\CLE^{\D}_\kappa$ ({resp.} $\CLE^{\C}_\kappa$). We {denote by} $\frk l_\eta$ the quantum length of $\eta$. Given a sample $\Gamma$ from $\CLE^{\C}_\kappa$ ({resp.} $\CLE^{\D}_\kappa$), we {use} $\mathrm{Count}_{\Gamma}(d\eta)$ ({resp.} $\mathrm{Count}^{'}_{\Gamma}(d\eta)$) {for} the counting measure on loops ({resp. outermost loops}) in $\Gamma$.

To construct the quantum $n$-hole surface $\QH_n(\ell_1,{\ldots},\ell_n)$, we follow {the same idea as for} the annulus and pair of pants. We first need the following disintegration lemma.
{When $n=2$, every family and product indexed by $1\le i\le n-2$ below is understood to be empty.}

\begin{lemma}
{
Fix an integer $n\ge2$ and $\ell_1>0$. Sample
$(D,h,\Gamma,z)/{\sim_\gamma}$ from
$\QD_{1,0}(\ell_1)\otimes\CLE_\kappa$, let $\eta_0$ be the outermost
loop surrounding $z$, and, with respect to
\[
\mathbf1_{\{\eta_0,\eta_1,\ldots,\eta_{n-2}\text{ are pairwise distinct}\}}
\prod_{i=1}^{n-2}\mathrm{Count}'_\Gamma(d\eta_i),
\]
set
$(\ell_2,\ldots,\ell_n):=(\mathfrak l_{\eta_0},\mathfrak l_{\eta_1},
\ldots,\mathfrak l_{\eta_{n-2}})$. For every Lebesgue-null Borel set
$E\subset(0,\infty)^{n-1}$,
\[
\QD_{1,0}(\ell_1)\otimes\CLE_\kappa
 [ (\ell_2,\ldots,\ell_n)\in E]=0.
\]
}
\end{lemma}
\begin{proof}
{
Use the area-pointing identity from Definition~\ref{def-QD}. After the bulk
point is integrated out, choosing $\eta_0$ amounts to weighting the ordered
factorial counting measure on pairwise distinct original outermost loops by
the non-negative random factor $\mu_h(D_{\eta_0})$; the CLE carpet has zero
quantum area almost surely. By Proposition~\ref{prop-ccm}, the
outermost-loop lengths are the jump sizes of the stated L\'evy process, under
an absolutely continuous reweighting of its law. Iterated Campbell--Mecke
for its Poisson jump measure shows that the ordered factorial measure of
every finite tuple of distinct jump sizes is absolutely continuous with
respect to Lebesgue measure. Stopping at $\tau_{-\ell_1}$, reweighting by
$\tau_{-\ell_1}^{-1}$, and multiplying by $\mu_h(D_{\eta_0})$ preserve null
sets. The claim follows.
}
\end{proof}
{To give a brief overview, our starting point is a quantum disk with one marked point and boundary length $\ell$. We remove the outermost loop that surrounds the marked point and then select $(n-2)$ additional, pairwise distinct outermost loops in the annular region under the counting measure. This yields a surface topologically equivalent to a sphere with $n$ holes. More specifically, we have the following definition and conformal welding formula.}

\begin{definition}\label{def: nholes}
{Fix an integer $n\ge2$.} For $\ell>0$, let $(\D,h,\Gamma,z)/{\sim_\gamma}$ be a sample from $\QD_{1,0}(\ell)\otimes\CLE^{\D}_\kappa(d\Gamma)$. Let $\eta_0$ be the outermost loop of $\Gamma$ surrounding $z$. {With respect to
\[
\mathbf1_{\{\eta_0,\eta_1,\ldots,\eta_{n-2}\,\text{are pairwise distinct}\}}
\prod_{i=1}^{n-2}\mathrm{Count}'_\Gamma(d\eta_i),
\]
choose the distinguished original outermost loops $\eta_1,\ldots,\eta_{n-2}$.} Let $H_{\eta_0, \eta_{1}, \textcolor{red}{\ldots} ,\eta_{n-2}}$ be the non-simply-connected component of $\D \setminus (\cup_{i=0}^{n-2} \eta_i)$. Let $\wt \QH_n(\ell,\ell_0,\ell_1,\textcolor{red}{\ldots},\ell_{n-2})$ be the law of the quantum surface $(H_{\eta_0, \eta_{1}, \textcolor{red}{\ldots} ,\eta_{n-2}},h)/{\sim_\gamma}$ disintegrated over $\frk l_{\eta_0}=\ell_0, \frk l_{\eta_1}=\ell_1,\textcolor{red}{\ldots},\frk l_{\eta_{n-2}}=\ell_{n-2}$. Let

\begin{equation}\label{equ:factor}
\QH_n(\ell,\ell_0,\ell_1,\textcolor{red}{\ldots},\ell_{n-2}) = {\ell_0}^{-1}{|\QD_{1,0}(\ell_0)|}^{-1}\left(\prod_{i=1}^{n-2} \ell_i |\QD(\ell_i)| \right)^{-1} \wt \QH_n(\ell,\ell_0,\ell_1,{\ldots},\ell_{n-2}).
\end{equation}

We call a sample from $\QH_n(\ell,\ell_0,\ell_1,{\ldots},\ell_{n-2})$ a quantum {$n$-hole} sphere.
\end{definition}

\begin{remark}
    In \cite{Int-CLE}, the authors introduce {the} quantum annulus $\QA(a,b)$ and {the} quantum pair of pants $\QP(a,b,c)$. We have $\QA(a,b)=\QH_2(a,b)$ and $\QP(a,b,c)=C(\gamma)\cdot \QH_3(a,b,c)$, where $C(\gamma)$ is given in Proposition{~}\ref{disk-welding}. 
\end{remark}
\begin{proposition}[Area distribution for $\QH_2$ and $\QH_3$ \cite{Int-CLE}]\label{thm-QA2-area}
\begin{itemize}
	\item[1.]For $a>0,b>0$, let $A$ be the quantum area of a sample from $\QA(a, b)$. Then
	\[\QH_2(a, b)[e^{-\mu A}] = \frac{\cos(\pi (\frac4{\gamma^2}-1))}\pi  \cdot \frac {e^{-(a+b)\sqrt{\mu/\sin(\pi\gamma^2/4)}} }{\sqrt{ab} (a+b)} \quad \textrm{for }\mu\ge 0.\]
\item[2.]For $\ell_1, \ell_2, \ell_3>0$, let $A$ be the total quantum area of a sample from $\QH_3(\ell_1, \ell_2, \ell_3)$.
Then there is a constant {$C_3(\gamma)\in(0,\infty)$, depending only on $\gamma$,} such that
\[\QH_3(\ell_1, \ell_2, \ell_3) [e^{-\mu A}] = C_3(\gamma)\mu^{\frac14 - \frac2{\gamma^2}}\frac{1}{\sqrt{\ell_1\ell_2\ell_3}} e^{-(\ell_1+ \ell_2+ \ell_3)\sqrt{\mu/\sin(\pi \gamma^2/4)}} \quad \textrm{for }\mu>0. \]
\end{itemize}

\end{proposition}
We will compute the constant $C_3(\gamma)$ in Theorem \ref{area-ufunc}.

\begin{proposition}\label{welding1 n-holes}
{Suppose we are in the setting of Definition~\ref{def: nholes}. The law of the decorated quantum surface $(\D,h,z,\eta_0, \eta_{1}, \dots ,\eta_{n-2})/{\sim_\gamma}$ equals}
\begin{align}
    \int_{{(0,\infty)^{n-1}}} \mathrm{Weld}(\QH_n(\ell,\ell_0,\ell_1,{\ldots},\ell_{n-2}), \QD_{1,0}(\ell_0), \QD(\ell_1),{\ldots},\QD(\ell_{n-2}))\prod_{j=0}^{n-2}\ell_j d\ell_j.
\end{align}

{In particular, the welding constant in this identity is $1$.}
\end{proposition}

\begin{theorem}\label{thm: welding QSn}
Assume $n\ge3$. Let 
	       $${\mathfrak C_n}=\QS_n\otimes\CLE_\kappa^{\C}\mathrm{1}_{\eta_i \text{ is the outermost loop surrounding } z_i \text{ and there is no noncontractible loop in } D_{\boldsymbol{\eta}}} .$$
        Then
        \begin{align}
\mathfrak C_n=C(\gamma)\int_{{(0,\infty)^n}}\mathrm{Weld}(\QH_n(\ell_1,\ell_2,{\ldots},\ell_n),\QD_{1,0}(\ell_1),{\ldots},\QD_{1,0}(\ell_n))\prod_{j=1}^n\ell_j d\ell_j,
        \end{align}
 where $C(\gamma)$ is given in Proposition~\ref{disk-welding}. 
\end{theorem}

{
We postpone the proofs of Proposition~\ref{welding1 n-holes} and Theorem~\ref{thm: welding QSn} until after Lemmas~\ref{symmetric point of QHn} and~\ref{removing points}, since they provide a new perspective on the quantum $n$-hole sphere. Recall the notation in Proposition~\ref{disk-welding}. We sample $(\C,h,\Gamma,0,\infty)/{\sim_\gamma}$ from $\QS_2 \otimes \CLE^{\C}_\kappa$. Sample $\eta$ from $1_E\mathrm{Count}_{\Gamma}(d \eta)$, where $E:=\{\eta\text{ separates }0\text{ and }\infty\}$. Let $D_{\eta}$ and $D^c_{\eta}$ be the two components of $\widehat\C\setminus\eta$, where $D_{\eta}$ contains $0$. 
Let $\Gamma|_{D_\eta}$ (resp. $\Gamma|_{D^c_\eta}$) be the subset of $\Gamma$ comprising loops contained in $D_\eta$ (resp. $D^c_\eta$). Let $\eta_0$ be the outermost loop in $\Gamma|_{D_\eta}$ surrounding $0$, and let $A_{\eta,\eta_0}$ be the annular region bounded by $\eta$ and $\eta_0$. With respect to $\prod_{i=1}^{n-2}\mathrm{Count}'_{\Gamma|_{A_{\eta,\eta_0}}}(d\eta_i)$, choose $\eta_1,\ldots,\eta_{n-2}$ and restrict to $F:=\{\eta_1,\ldots,\eta_{n-2}\text{ are pairwise distinct}\}$. For each $i$, let $D_{\eta_i}$ be the simply connected component of $\widehat\C\setminus\eta_i$ contained in $A_{\eta,\eta_0}$. Conditioning on $\{\eta_1,\ldots,\eta_{n-2}\}$, sample $z_i$ according to $\mu_h(dz_i)|_{D_{\eta_i}}$.
}

\begin{lemma}\label{symmetric point of QHn}
 {Following the notation and procedure above, let $\mathbb F$ be a measure on $H^{-1}(\C)$ such that the law of $(\C,h,0,\infty)/{\sim_\gamma}$ is $\QS_2$ if $h$ is sampled from $\mathbb F$. Under the measure}
    \begin{equation}\label{eq:measure QSn}
   {\mathbf1_F\mathbf1_E} \prod_{i=1}^{n-2}\left[ \mu_h(dz_i)|_{D_{\eta_i}}{\mathrm{Count}^{'}_{\Gamma|_{A_{\eta,\eta_0}}}(d\eta_i)} \right]\, \mathbb F(dh)\,\mathrm{Count}_{\Gamma}(d \eta)\CLE_\kappa^\C(d\Gamma).
    \end{equation}
    The law of decorated quantum surface 
    $$
    (\C,h,0,\infty,z_{1}, \dots, z_{n-2}, \eta_0, \eta, \eta_{1}, \dots, \eta_{n-2})/{\sim_\gamma}
    $$
    is ${\mathfrak C_n}$.
\end{lemma}
\begin{proof}
    {The main idea is to interchange the order of integration in~\eqref{eq:measure QSn}. For fixed $(h,\Gamma,0,\infty)$ and $\eta$, summing over the ordered, pairwise distinct original outermost loops in $A_{\eta,\eta_0}$ and then integrating the $z_i$ gives the product area measure restricted to the event that the points lie in distinct outermost-loop interiors. Summing next over the adjacent separating pair $(\eta,\eta_0)$ partitions exactly the configurations for which all retained loops are contractible relative to $(0,\infty,z_1,\ldots,z_{n-2})$. Indeed, the labelled retained-loop tuple uniquely recovers $(\eta,\eta_0,\eta_1,\ldots,\eta_{n-2})$, so each configuration occurs with multiplicity one in the factorial counting measure. Boundaries and the CLE carpet have zero $\mu_h$-area almost surely, so no residual area term appears. Thus the marginal point measure is the corresponding restriction of $\prod_{i=1}^{n-2}\mu_h(dz_i)$, and, conditionally on the marked points, $\eta,\eta_0,\eta_1,\ldots,\eta_{n-2}$ are precisely their associated outermost loops. This is the claimed restricted pushforward of $\QS_n\otimes\CLE_\kappa^\C$.}
\end{proof}

{We also need a variant of Lemma~\ref{symmetric point of QHn}, obtained by removing one marked point.}

\begin{lemma}\label{removing points}
    Retain the setting of Lemma~\ref{symmetric point of QHn}. {For fixed $i\in\{1,\ldots,n-2\}$, let $M$ be the measure obtained from~\eqref{eq:measure QSn} by deleting the factor $\mu_h(dz_i)|_{D_{\eta_i}}$ and forgetting $z_i$. Thus $M$ describes the decorated quantum surface}
    $$
    (\C,h,0,\infty,z_{1}, \dots,\hat{z_i},\dots, z_{n-2}, \eta_0, \eta, \eta_{1}, \dots, \eta_{n-2})/{\sim_\gamma}.
    $$
  {Then, conditionally on the exterior decorated surface $(D^{c}_{\eta_i}, h,0,\infty, z_1,\dots,\widehat z_i,\dots, z_{n-2})/{\sim_\gamma}$, the law of $(D_{\eta_i},h)/{\sim_\gamma}$ is $\QD(\mathfrak l_{\eta_i})^\#$.}
\end{lemma}

\begin{proof}
{By Lemma~\ref{symmetric point of QHn}, the marked points are exchangeable under the restricted $\QS_n$ measure. Proposition~\ref{disk-welding} and the quantum-sphere resampling property imply that, conditionally on the exterior data and the boundary length $\mathfrak l_{\eta_i}$, the pointed component $(D_{\eta_i},h,z_i)/{\sim_\gamma}$ has law $\QD_{1,0}(\mathfrak l_{\eta_i})^\#$. Deleting the factor $\mu_h(dz_i)|_{D_{\eta_i}}$ exactly unbiases by quantum area and forgets the bulk-typical point. By Definition~\ref{def-QD}, the resulting conditional law is $\QD(\mathfrak l_{\eta_i})^\#$, as claimed.}
\end{proof}

{We now prove Proposition~\ref{welding1 n-holes} and Theorem~\ref{thm: welding QSn}.}

\begin{proof}[Proof of Proposition~\ref{welding1 n-holes}]
{Disintegrate the sampled configuration of Definition~\ref{def: nholes} over the boundary lengths $\mathfrak l_{\eta_0},\ldots,\mathfrak l_{\eta_{n-2}}$. By Proposition~\ref{prop-msw} and the CLE domain Markov property, conditionally on these lengths and on the central component, the components cut out by $\eta_0,\eta_1,\ldots,\eta_{n-2}$ are independent and have respective laws $\QD_{1,0}(\ell_0)^\#$ and $\QD(\ell_i)^\#$, $1\le i\le n-2$. Hence the disintegrated measure is
\[
\wt\QH_n(\ell,\boldsymbol\ell)
\otimes\QD_{1,0}(\ell_0)^\#
\otimes\bigotimes_{i=1}^{n-2}\QD(\ell_i)^\#.
\]
Substituting the definition~\eqref{equ:factor}, together with $\QD_{1,0}(\ell_0)^\#=|\QD_{1,0}(\ell_0)|^{-1}\QD_{1,0}(\ell_0)$ and $\QD(\ell_i)^\#=|\QD(\ell_i)|^{-1}\QD(\ell_i)$, cancels all total-mass factors and leaves $\prod_{j=0}^{n-2}\ell_j\,d\ell_j$. This is exactly the displayed welding identity, with constant $1$.}
\end{proof}

\begin{proof}[Proof of Theorem~\ref{thm: welding QSn}]
 {Apply Proposition~\ref{disk-welding} to the separating loop $\eta$, then apply Proposition~\ref{welding1 n-holes} inside the component containing $0$. For $1\le i\le n-2$, the factor $\mu_h(dz_i)|_{D_{\eta_i}}$ converts $\QD(\ell_i)$ exactly into $\QD_{1,0}(\ell_i)$ by Definition~\ref{def-QD}, so no further constant occurs. Lemma~\ref{symmetric point of QHn} identifies the resulting retained-loop configuration with the restricted $\QS_n\otimes\CLE_\kappa^\C$ measure. The first welding contributes $C(\gamma)$, while Proposition~\ref{welding1 n-holes} has constant $1$, giving the stated formula.}
\end{proof}

\subsection{Area distribution of {a} quantum $n$-{hole} sphere}\label{subsec:partition}

The goal of this section is to compute $\QH_n(\ell_1,\ell_2,{\ldots},\ell_n)[e^{-\mu A}]$. {For $2\leq j\leq n$, set $L_j:=\sum_{q=1}^j\ell_q$.} We first recall some background in L\'evy processes. Let  
\begin{equation}
\lexp= \frac4{\kappa} + \frac12=\frac{4}{\gamma^2}+\frac12 \in (\frac32,2).
\end{equation}
Let $\P^{\beta}$ be the probability measure on 
c\`adl\`ag processes on $[0,\infty)$ describing the law of a $\lexp$-stable L\'evy process  
with L\'evy measure $1_{x>0} x^{-\beta-1} \, dx$.
Let $(\zeta'_t)_{t\geq 0}$ be a sample from $\P^{\beta}$.
Let $J:=\{(x,t): t\ge 0 \textrm{ and }\zeta'_t-\zeta'_{t^-}=x>0 \}$ be the set of jumps of $\zeta'$. Given  $J$, let $(\bfb, \bft)\in J$ be sampled from the counting measure on $J$.
Let $M^{\beta}$ be the law of  $(\zeta', \bfb,\bft)$.
Namely, $M^{\beta}$ is the infinite measure such that for non-negative measurable functions $f,g$ we have
\[
M^{\beta}[f(\zeta')g(\bfb,\bft)] =\int \big(f(\zeta')\sum_{(x,t)\in J} g(x,t) \big)d\P^{\beta}.
\]
For each $a>0$, let $\tau_{-a}=\inf\{ t: \zeta'_t=-a\}$ and $J_a=\{(x,t)\in J: t\le \tau_{-a}\}$.



{In order to prove Lemma~\ref{lem:disint QHn}, we use the following Campbell--Palm description of $M^\beta$.}

\begin{lemma}\label{lem:palm}
	Let $\wt\zeta'_t=\zeta'_t- 1_{t\ge \bft} \bfb$. Then the $M^\beta$-law of $(\bfb,\bft)$ is $1_{b>0,t>0} b^{-\beta-1} \, db \, dt$, and 
   {conditionally} on $(\bfb,\bft)$, the conditional law of $\wt\zeta'_t$ under $M^\beta$ is ${\P^\beta}$. 
	Equivalently, the joint law of $(\bfb,\bft)$ and $(\wt\zeta'_t)_{t\ge 0}$ under $M^\beta$ is the product measure $(1_{b>0,t>0} b^{-\beta-1} \, db \, dt) \times {\P^\beta}$.  
\end{lemma}
\begin{proof}
{By the definition of the L\'evy measure, the atoms driving the jumps of $\zeta'$ form a Poisson random measure on $(0,\infty)^2$ with intensity $1_{\{x>0,t>0\}}x^{-\beta-1}\,dx\,dt$. Campbell's formula gives the stated $M^\beta$-law of $(\bfb,\bft)$. By the compensated L\'evy--It\^o representation and the Slivnyak--Mecke theorem, deleting the marked atom from the compensated Poisson random measure leaves a Poisson random measure with the original intensity, independent of $(\bfb,\bft)$. Hence the full process $(\widetilde\zeta'_s)_{s\ge0}$ has law $\P^\beta$ and is independent of $(\bfb,\bft)$.}
\end{proof}

{Recall that Proposition~\ref{lem:disint QH3} connects the L\'evy process with the disintegration law of the outermost-loop lengths. We denote the boundary of $\QD_{1,0}(\ell_1)$ by $\eta_1$, with $\ell_h(\eta_1)=\ell_1$.}

\begin{proposition}[Lemma 7.11 \cite{Int-CLE}]\label{lem:disint QH3}
    For $\ell_1>0$, let $(D,h,\Gamma,z)/{\sim_\gamma}$ be a sample from $\QD_{1,0}(\ell_1)\otimes\CLE_\kappa${, and let} $\eta_2$ be the outermost loop surrounding $z$. 
Let $\eta_3$ be chosen from the counting measure on the outermost loops of $\Gamma$ except $\eta_2$. 
{Let $(x^{(3)}_i)_{i\ge 1}$ be the lengths of the original outermost loops of $\Gamma$ other than $\eta_2$ and $\eta_3$, ranked in decreasing order.}
Then for $\ell_2,\ell_3>0$, the disintegration of the law of $\{ x^{(3)}_i: i\ge 1 \}$ over $\{\ell_h(\eta_2)=\ell_2,\ell_h(\eta_3)=\ell_3\}$ is the law of 
$\{x: (x,t)\in  J_{\ell_1+\ell_2+\ell_3}  \textrm{ for some }t  \}$ under
\[
\ell_2|\QA (\ell_1,\ell_2) | |\QD_{1,0}(\ell_2)|   {\ell_3}^{-\beta-1} \tau_{-\ell_1-\ell_2}  \P^\beta(d\zeta').
\]
\end{proposition}

{Proposition~\ref{lem:disint QH3} extends to the multiple-loop case. For a sequence $(\ell_i)_{i\ge1}$ of positive numbers, we use the following notation.}

Let $\{P_k, k \ge 3\}$ denote the quantities {defined by the following recursion}: 

Starting at $k=3$, $P_3=\tau_{-\ell_1-\ell_2},$
\begin{align*}
P_4=&\tau_{-\ell_1-\ell_2-\ell_4}\cdot\tau_{-\ell_1-\ell_2}+\tau_{-\ell_1-\ell_2}\cdot(\tau_{-\ell_1-\ell_2-\ell_3}-\tau_{-\ell_1-\ell_2})\\
=&\tau_{-\ell_1-\ell_2}\cdot\tau_{-\ell_1-\ell_2-\ell_3}+\tau_{-\ell_1-\ell_2}\cdot\tau_{-\ell_1-\ell_2-\ell_4}-\tau_{-\ell_1-\ell_2}\cdot\tau_{-\ell_1-\ell_2}.
\end{align*}

{For $n\ge5$,} to define $P_n$, we focus on each monotone item of $P_{n-1}$. {Every such item has the form
$\tau_{-s_1}\tau_{-s_2}\cdots\tau_{-s_{n-3}}$, where
\[
s_j=L_2+\sum_{i\in A_j}\ell_i,\qquad
\varnothing=A_1\subseteq A_2\subseteq\cdots\subseteq A_{n-3}
\subseteq\{3,\ldots,n-1\}.
\]
Equal consecutive subsets, and hence equal consecutive thresholds, are allowed.}

{Extend the following transformation linearly, preserving the sign and multiplicity of each input occurrence.} A monotone item of the form $\tau_{-s_1}\cdot\tau_{-s_2}\cdots\tau_{-s_{n-3}}$ transforms to the following combination:
\begin{align*}
    &\quad\tau_{-s_1}\cdot\tau_{-s_1-{\ell_n}}\cdot\tau_{-s_2-{\ell_n}}{\cdots}\tau_{-s_{n-3}-{\ell_n}}\\
    &+(\tau_{-s_2}-\tau_{-s_1})\cdot\tau_{-s_1}\cdot\tau_{-s_2-{\ell_n}}{\cdots}\tau_{-s_{n-3}-{\ell_n}}\\
    &+\cdots\\
    &+(\tau_{-s_{n-3}}-\tau_{-s_{n-4}})\cdot\tau_{-s_1}\cdot\tau_{-s_2}{\cdots}\tau_{-s_{n-3}-{\ell_n}}\\
    &+(\tau_{-L_{n-1}}-{\tau_{-s_{n-3}}})\cdot\tau_{-s_1}\cdot\tau_{-s_2}{\cdots}\tau_{-s_{n-3}}.
\end{align*}
{After expanding the brackets and reordering the factors, the result is still a linear combination of monotone terms.}
{For $n=5$, the second and penultimate displayed lines denote the same term and are included only once.}

Therefore, the {sequence} $\{P_k:k\ge3\}$ is well defined. {Let $\mathcal P_n$ be the multiset of signed monomial occurrences in the fully expanded expression for $P_n$, before collecting equal terms.}

\begin{lemma}\label{lem:disint QHn}
  {Let $n\ge3$ and $\ell_1,\ldots,\ell_n>0$.} Let $(D,h,\Gamma,z)/{\sim_\gamma}$ be a sample from $\QD_{1,0}(\ell_1)\otimes\CLE^{\D}_\kappa$, and set {$\eta_1:=\partial D$}. Let $\eta_2$ be the outermost loop surrounding $z$. {With respect to the product counting measure restricted to pairwise distinct loops, choose $\eta_3,\ldots,\eta_n$ among the original outermost loops of $\Gamma$ other than $\eta_2$.}
Let $(x^{(n)}_i)_{i\ge 1}$ be the lengths of the {original outermost loops of $\Gamma$ other than $\eta_2,\ldots,\eta_n$}, ranked in decreasing order. Then the disintegration of the law of $\{ x^{(n)}_i: i\ge 1 \}$ over $\{\ell_h(\eta_i)=\ell_i, {i=2,\ldots,n}\}$ is the law of 
$\{x: (x,t)\in  J_{L_n}  \textrm{ for some }t  \}$ under
\[
\ell_2|\QA (\ell_1,\ell_2) | |\QD_{1,0}(\ell_2)|   \prod_{i=3}^{n}{\ell_i}^{-\beta-1} P_n  \P^\beta(d\zeta').
\]
\end{lemma}
\begin{proof}
  We first let $(x^{(2)}_i)_{i\ge 1}$ be the lengths of the outermost loops of $\Gamma$ except $\eta_2$, ranked in decreasing order. 
 The disintegration law of {$\{x_i^{(2)}:i\ge1\}$} over $\{\ell_h(\eta_2)=\ell_2\}$ is the law of 
$\{x: (x,t)\in  J_{L_2}  \textrm{ for some }t  \}$ under
\[
\ell_2|\QA (\ell_1,\ell_2)||\QD_{1,0}(\ell_2)|\P^\beta(d\zeta').
\]
  {Write $\wt\tau_{-a}:=\inf\{t\geq0:\wt\zeta'_t=-a\}$. Under $M^\beta$ restricted to $\{(\bfb,\bft)\in J_{L_2}\}$, Lemma~\ref{lem:palm} gives the joint measure
  \[
  1_{\{b>0\}}b^{-\beta-1}\,db\,
  1_{\{0<t<\wt\tau_{-L_2}\}}\,dt\,
  \P^\beta(d\wt\zeta').
  \]
  After disintegrating at $b=\ell_3$, the unmarked jumps are those before $\wt\tau_{-L_3}$. Under deletion of the marked jump, $\{\bft<\tau_{-a}\}$ for $\zeta'$ becomes $\{\bft<\wt\tau_{-a}\}$ for $\wt\zeta'$, and $\tau_{-a}$ becomes $\wt\tau_{-(a+b)}$.}

  {Therefore, for $n=3$, the disintegration law of
  $\{x_i^{(3)}:i\geq1\}$, the lengths of the outermost loops of $\Gamma$
  other than $\eta_2$ and $\eta_3$, is the law of the jumps under the measure}
  \begin{align*}
   \ell_2|\QA (\ell_1,\ell_2) | |\QD_{1,0}(\ell_2)|{\ell_3}^{-\beta-1}P_3\P^\beta(d\zeta'),
  \end{align*}
  {where $P_3=\tau_{-\ell_1-\ell_2}$ and only the jumps before
  $\tau_{-L_3}$ are retained.}
  
{For $n=4$, we use the same argument to remove $\eta_4$.}
{We first sample $\ell_4$ under the measure $1_{\{\ell_4>0\}}\ell_4^{-\beta-1}\,d\ell_4$ and then fix $\ell_4$. On $\{\bft<\tau_{-L_2}\}$, deleting the marked jump and disintegrating at $\bfb=\ell_4$ gives the density $\ell_4^{-\beta-1}\tau_{-L_2}\tau_{-\ell_1-\ell_2-\ell_4}\P^\beta(d\zeta')$. On $\{\tau_{-L_2}<\bft<\tau_{-L_3}\}$ it gives $\ell_4^{-\beta-1}(\tau_{-L_3}-\tau_{-L_2})\tau_{-\ell_1-\ell_2}\P^\beta(d\zeta')$. Summing the two contributions, we get the $n=4$ disintegration formula}
  \begin{align*}
      \ell_2|\QA (\ell_1,\ell_2) | |\QD_{1,0}(\ell_2)|   {\ell_3}^{-\beta-1}{\ell_4}^{-\beta-1} P_4  \P^\beta(d\zeta'),
  \end{align*}
  
 {In both displays, only the jumps before the indicated terminal level $L_k$ are retained. For general $n$, assume the disintegration law for $n-1$ and consider $P_{n-1}\P^\beta(d\zeta')$. For each monotone item of the form $\tau_{-s_1}\cdots\tau_{-s_{n-3}}$, split according to $\{\bft<\tau_{-s_1}\}$, $\{\tau_{-s_1}<\bft<\tau_{-s_2}\}$, $\ldots$, and $\{\tau_{-s_{n-3}}<\bft<\tau_{-L_{n-1}}\}$. Applying the same argument as for $n=4$ gives precisely the recursion defining $P_n$, which concludes the proof.}

{More precisely, the successive Campbell--Palm argument
  above remains valid after multiplication by an arbitrary bounded
  non-negative measurable functional $H(\zeta')$. It therefore identifies
  $P_n(\zeta')\,\P^\beta(d\zeta')$ itself, and not merely its ranked-jump
  pushforward, with the positive disintegration measure on L\'evy path
  space. Consequently, $P_n\geq0$ $\P^\beta$-almost surely.}

\end{proof}

Now we can formulate the following theorem, which describes the area distribution of a quantum {$n$-hole} sphere.

\begin{theorem}\label{thm:area QHn}
{Let $n\ge3$, $\ell_1>0$, and $\mu>0$, and let $A$ denote quantum area. Define $g:(0,\infty)\to(0,1]$ by $g(x)=\E[e^{-\mu x^2S_1}]$, where $S_1$ is the quantum area of a sample from $\QD(1)^{\#}$. Then, for Lebesgue-almost-every $(\ell_2,\ldots,\ell_n)\in(0,\infty)^{n-1}$,}
\begin{align} \label{areacompute}
\QH_n(\ell_1,\ell_2,{\ldots},\ell_n)[e^{-\mu A}]=&\frac{1}{R_\gamma^{n-2}}|\QA(\ell_1,\ell_2)|\prod_{j=3}^n \ell_j^{-\frac{1}{2}}\E^{\beta}[\tau_{-L_2}\tau_{-L_n}^{n-3}\prod_{i\geq 1}g(x_i)]
\end{align}
{Here $(x_i)_{i\geq1}$ is the sequence of positive jump sizes before $\tau_{-L_n}$.}
\end{theorem}
{In particular, the theorem asserts that the right-hand side is finite.}

Before the proof, we first give {an} explicit formula for $\QH_n(\ell_1,\ell_2,{\ldots},\ell_n)[e^{-\mu A}]$ {in terms of} $P_n$ in Lemma~\ref{lem-QHn-translate}.


\begin{lemma}\label{lem-QHn-translate}
{Let $n\ge3$, $\ell_1>0$, and $\mu>0$. For
Lebesgue-almost-every $(\ell_2,\ldots,\ell_n)\in(0,\infty)^{n-1}$, with
$g$ as in Theorem~\ref{thm:area QHn}, we have}
	\begin{equation}\label{eq:QHn levy}
    \QH_n(\ell_1,\ell_2,{\ldots},\ell_n)[e^{-\mu A}]=\frac{1}{R_\gamma^{n-2}}|\QA(\ell_1,\ell_2)|\prod_{j=3}^n \ell_j^{-\frac{1}{2}}\E^{\beta}[P_n\prod_{i\geq 1}g(x_i)]
	\end{equation}
{Here the expectation $\E^\beta$ is with respect to $\P^\beta$ and $(x_i)_{i\ge1}$ is the sequence of positive jump sizes before $\tau_{-L_n}$.}
\end{lemma}

\begin{proof}
{Sample $(D,h,\Gamma,z)/{\sim_\gamma}$ from 
$\QD_{1,0}(\ell_1)\otimes\CLE^{\D}_\kappa$, set $\eta_1:=\partial D$, let $\eta_2$ be the outermost loop surrounding $z$, and then choose the pairwise distinct original outermost loops $\eta_3,\ldots,\eta_n$, all different from $\eta_2$, under the product counting measure. Let $H_n=H_{\eta_1,\ldots,\eta_n}$ be the region bounded by $\partial D=\eta_1$ and $\eta_2,\ldots,\eta_n$. Lemma~\ref{lem:disint QHn} gives the law of the quantum lengths. By Proposition~\ref{prop-msw}, conditionally on the lengths, the components surrounded by the original outermost loops are independent quantum disks. Since the CLE carpet has zero $\mu_h$-area almost surely, the area of $H_n$ is the sum of the areas of the unselected disk components.}
Therefore the integral of $e^{-\mu {\operatorname{Area}}(H_{n})}$ over this sample space is 
\begin{equation}\label{eq:case1}
{\int_{(0,\infty)^{n-1}}}\ell_2|\QA (\ell_1,\ell_2) | |\QD_{1,0}(\ell_2)|    \E^{\beta}[P_n\prod_{i\geq 1}g(x_i)]\prod_{i=3}^{n}{\ell_i}^{-\beta-1} {d\ell_2\cdots d\ell_n}.
\end{equation}
By Proposition~\ref{welding1 n-holes}, {the integral in~\eqref{eq:case1} equals}
\[
{\int_{(0,\infty)^{n-1}} \QH_n(\ell_1,\ell_2,\ldots,\ell_n)[e^{-\mu\operatorname{Area}(H_n)}]\, |\QD_{1,0}(\ell_2)|\prod_{k=3}^n |\QD(\ell_k)|\prod_{j=2}^n\ell_j\,d\ell_2\cdots d\ell_n.}
\]
{Since $|\QD(\ell)|=R_\gamma\ell^{-2-\frac{4}{\gamma^2}}$, we obtain~\eqref{eq:QHn levy} after disintegrating over $\{\ell_h({\eta_i})=\ell_i, i=2, \ldots, n\}$.}

\end{proof}

{To compute the L\'evy-process expectation $\E^{\beta}[P_n\prod_{i\geq 1}g(x_i)]$, we use a combinatorial argument to reduce it to the symmetric form $\E^{\beta}[\tau_{-L_2}\tau_{-L_n}^{n-3}\prod_{i\geq 1}g(x_i)]$, which depends only on $L_2$ and $L_n$.}

Now, we recall the excursion theory for the L\'evy process $\zeta'$.
Let $I_t= \inf\{\textcolor{red}{\zeta'_s}: s\in [0,t] \}$. It is well-known that $(\zeta'_t-I_t)_{t\ge 0}$ is a Markov process and we can consider its excursions away from 0.
See e.g.~\cite[Section 1]{duquesne-legall}, ~\cite[Section 4]{nested-stat}. We write $\ul N'$ for the excursion measure, which is a measure on non-negative functions of the form  
$e:[0,T] \rta [0,\infty)$ with $e(0)=e(T)=0$. 

{More specifically, let $N_1'$ be a probability measure on excursions of duration 1. Let $N_\ell'$ be the law of $(\ell^{1/\beta}b_{t/\ell}:0\le t\le\ell)$, extended by $0$ after time $\ell$, where $b$ is sampled from $N_1'$; thus $N_\ell'$ is the probability law of an excursion of duration $\ell$. Writing $G_\beta:=\Gamma(-\beta)$, the Laplace exponent of $\zeta'$ is $G_\beta q^\beta$. Since $y\mapsto\tau_{-y}$ is therefore a stable subordinator with Laplace exponent $G_\beta^{-1/\beta}q^{1/\beta}$,}
\begin{align}
		{\ul N' = \frac{G_\beta^{-1/\beta}}{\beta\Gamma(1-1/\beta)} \int_0^\infty \ell^{-1-1/\beta}\,N_{\ell}'\,d\ell.}
\end{align}
{Choose a c\`adl\`ag modification $y\mapsto\widehat\tau_y$ of the inverse local-time process $y\mapsto\tau_{-y}$.  Thus, for every fixed deterministic $y$, $\widehat\tau_y=\tau_{-y}$ almost surely; simultaneous equality for every $y$ is not asserted.  Then}
\begin{align*}
		{\mathcal E:=\sum_{y>0\colon \widehat\tau_y > \widehat\tau_{y-}}
	\delta_{(y,\zeta'_{(\widehat\tau_{y-}+\cdot)\wedge \widehat\tau_y}
	-\zeta'_{\widehat\tau_{y-}})}
	\quad\text{is a P.P.P.\ with intensity}\quad
	\lambda|_{(0,\infty)}\otimes \ul N'}
\end{align*}
on $(0,\infty)\times D([0,\infty))$ where $\lambda$ is the Lebesgue measure on $\R$ and $D([0,\infty))$ is the space of càdlàg functions. We refer to \cite[Section 3.1.1]{curien-kortchemski-looptree-def} for more information on these constructions.

{
Recall that $(x_i)_{i\ge1}$ is the sequence of positive jump sizes of $\zeta'$ before $\tau_{-L_n}$. The excursions of $(\zeta'_t-I_t)_{t\in[0,\tau_{-L_n}]}$ form a Poisson point process $\mathrm{Exc}_{L_n}=\{(e,s):s\le L_n\}$ with intensity $\ul N'(de)1_{\{s\in[0,L_n]\}}\,ds$. For $(e,s)\in\mathrm{Exc}_{L_n}$, let $\sigma$ be the starting time of the excursion $e$. Then $s=-I_\sigma$. Let $T(e)$ be the duration of $e$. Then, simultaneously for $0\le\ell\le L_n$,
 \[\widehat\tau_{\ell} = \sum_{ (e,s)\in \mathrm{Exc}_{L_n}}  T(e)1_{\{s\le \ell\}}.\]
In particular, for every deterministic $\ell$ this identity also holds with $\widehat\tau_\ell$ replaced by $\tau_{-\ell}$, almost surely.
}
 
 As a general property of Poisson point processes, conditioning on $\{ e: (e,s)\in \mathrm{Exc}_{L_n} \}$, the conditional law of 
		the time set $\{ s: (e,s)\in \mathrm{Exc}_{L_n} \}$ is given by a collection of independent uniform random variables in $(0,L_n)$. 
		Let $\cF_{L_n}$ be the {$\sigma$-algebra} generated by $\{ e: (e,s)\in \mathrm{Exc}_{L_n} \}$ and denote the conditional expectation over $\cF_{L_n}$ by $\E_{L_n}$.{Let $J_e$ be the multiset of positive jump sizes of the excursion $e$, and set $F(e):=\prod_{x\in J_e}g(x)$.} Then $\prod_{i\ge 1}g(x_i)= \prod_{e\in \mathrm{Exc}_{L_n} }F(e)$. {Here $F(e)$ is the decreasing limit of the products over jumps of $e$ of size at least $\delta$ as $\delta\downarrow0$; products over infinitely many excursions are defined analogously by restricting first to excursions with $T(e)\geq\delta$.}

\begin{proposition}\label{prop:combinatorics}
For any $n\geq 3$,
    $$
    \E^{\beta}[P_n\prod_{i\geq 1}g(x_i)]=\E^{\beta}[\tau_{-L_2}\tau_{-L_n}^{n-3}\prod_{i\geq 1}g(x_i)]
    $$
\end{proposition}

\begin{proof}
Consider the excursion theory above. {$(x_i)_{i\ge1}$ is the sequence of positive jump sizes of $\zeta'$ before $\tau_{-L_n}$.} Let $F(e)=\prod_{x\in J_e} g(x) $, so  $\prod_{i\ge 1}g(x_i)= \prod_{e\in \mathrm{Exc}_{L_n} }F(e)$.
	 Since $\prod_{i\ge 1}g(x_i) $ is measurable with respect to $\cF_{L_n}$, we have 
	\begin{equation}\label{eq: stopping simple}
	    \E_{L_n}\big[P_n \prod_{i\ge 1}g(x_i) \big]= \E_{L_n}\big[P_n \prod_{e\in \mathrm{Exc}_{L_n} }F(e) \big]= \E_{L_n}\big[P_n\big]
	      \prod_{e\in \mathrm{Exc}_{L_n} }F(e).
	\end{equation}
	{Here the conditional expectations are understood in the
extended non-negative sense; the required non-negativity of $P_n$ was
established at the end of the proof of Lemma~\ref{lem:disint QHn}.}
	
	{
For the independent-increment comparison below, fix $\lambda>0$ and set
\[
Z_\lambda:=\E^\beta\!\left[\prod_{e\in\mathrm{Exc}_{L_n}}F(e)
 e^{-\lambda\tau_{-L_n}}\right],\qquad
\widehat\E_\lambda[H]
:=Z_\lambda^{-1}\E^\beta\!\left[H\prod_{e\in\mathrm{Exc}_{L_n}}F(e)
 e^{-\lambda\tau_{-L_n}}\right].
\]
The Poisson exponential formula, together with~{\cite[Proposition 6.7]{Int-CLE}}, gives
$0<Z_\lambda<\infty$. Let $\widehat\P_\lambda$ denote the probability measure
corresponding to $\widehat\E_\lambda$. By the multiplicative formula for
Poisson point processes, under $\widehat\P_\lambda$ the process
$\mathrm{Exc}_{L_n}$ is a Poisson point process with intensity
\[
1_{(0,L_n]}(s)\,ds\,F(e)e^{-\lambda T(e)}\ul N'(de).
\]
Hence $(\widehat\tau_a)_{0\le a\le L_n}$ has stationary independent increments under
$\widehat\P_\lambda$. Moreover, since $0\leq F(e)\leq1$ and the image of
$\ul N'$ under $e\mapsto T(e)$ has density $c_\beta t^{-1-1/\beta}\,dt$
for a constant $c_\beta>0$, for every integer $r\geq1$,
\[
\int T(e)^rF(e)e^{-\lambda T(e)}\ul N'(de)
\leq c_\beta\int_0^\infty t^{r-1-1/\beta}e^{-\lambda t}\,dt<\infty.
\]
Thus all polynomial moments used below are finite. It is enough to prove
\[
\widehat\E_\lambda[P_n]
=\widehat\E_\lambda[\tau_{-L_2}\tau_{-L_n}^{n-3}]
\]
for each $\lambda>0$. After multiplication by $Z_\lambda$, this is the equality
of the corresponding expectations with the additional factor
$e^{-\lambda\tau_{-L_n}}$. Letting $\lambda\downarrow0$ then proves the
proposition by monotone convergence, since
$e^{-\lambda\tau_{-L_n}}\uparrow1$ and both integrands are non-negative;
here $0\leq F(e)\leq1$, $\tau_{-L_2}\tau_{-L_n}^{n-3}\geq0$, and
$P_n\geq0$ by Lemma~\ref{lem:disint QHn}.
}
	
{
Define $f$ on every signed monomial occurrence
$M=\varepsilon\prod_{j=1}^{n-2}\tau_{-s_j}\in\mathcal P_n$, where
$s_1=L_2$, by
\[
\begin{aligned}
f&:\mathcal P_n\longrightarrow\{\pm1\}\times\N^{n-2},\\
f(M)&:=(\varepsilon;a_1,\ldots,a_{n-2}),
\end{aligned}
\]
where $a_j:=\#\{\text{summands in }s_j\}$ for $1\le j\le n-2$.
Since $s_1=L_2=\ell_1+\ell_2$, we have $a_1=2$.
We also write $(\varepsilon;a_1,\ldots,a_{n-2})$ as
$\varepsilon(a_1,\ldots,a_{n-2})$ and call $f(M)$ the structure number of
$M$.
}
	
 For instance, when $n=4$\textcolor{red}{,} the structure number of $\tau_{-(\ell_1+\ell_2)}\cdot\tau_{-(\ell_1+\ell_2+\ell_4)}$ is $(2,3)$, since $f(+\tau_{-(\ell_1+\ell_2)}\cdot\tau_{-(\ell_1+\ell_2+\ell_4)})= +(2,3)\textcolor{red}{.}$ 
	
Now let us state some properties about monotone items:
\begin{itemize}
	\item {First, let $\sigma$ be a permutation of $\{3,\ldots,n\}$. Define its action on $\tau_{-s_1}\cdots\tau_{-s_{n-2}}$ by replacing $\ell_i$ with $\ell_{\sigma(i)}$ for $i\ge3$ and fixing $\ell_1,\ell_2$. We extend this action linearly to the span of $\mathcal P_n$. Then, for every $M\in\mathcal P_n$, $f(M)=f(\sigma M)$ and $\sigma(\mathcal P_n)=\mathcal P_n$. This follows by induction.}
    \item {Second, the sum of these monotone items satisfies a recursion described by the binary tree in Figure~\ref{fig:structure-number-tree}, where $+$ and $-$ are the signs of the corresponding structure classes in $P_n$. Let $L$ and $R$ be the two operations that map each parent node to its left and right child nodes. For a signed structure number $\pm(2,a_2,\ldots,a_k)$ at level $P_{k+2}$, its left child at level $P_{k+3}$ is $L(\pm(2,a_2,\ldots,a_k))=\pm(2,3,a_2+1,\ldots,a_k+1)$, whereas its right child is $R(\pm(2,a_2,\ldots,a_k))=\mp(2,2,a_2',\ldots,a_k')$, where
	$$
	a_{j}'=
	\begin{cases}
	2 & a_{j}=2\\
	a_{j}+1 & a_{j}>2
	\end{cases}
	$$
	}
	{
To verify this whole-level regrouping, write $b_j=a_j-2$ and let
$c_d(0,b_2,\ldots,b_d)$ denote the coefficient of a nested subset chain
with these cardinalities in $P_{d+2}$; the following endpoint calculation
also shows inductively that this coefficient depends only on the
cardinalities. If the new label $u=d+3$ first occurs at the $p$th entry of
a child chain $\varnothing=C_1\subseteq\cdots\subseteq C_{d+1}$, the defining recursion
has only two possible sources: the upper endpoint, which requires
$C_p=C_{p-1}\cup\{u\}$, and the lower endpoint, which requires $p\geq3$
and $C_{p-1}=C_{p-2}$. These sources have opposite signs and cancel when
both conditions hold. If $u$ does not occur, the same calculation applies
to the two final endpoints. Consequently, with $\rho(0)=0$ and
$\rho(q)=q-1$ for $q\geq2$,
\[
c_{d+1}(0,b_2,\ldots,b_{d+1})=
\begin{cases}
c_d(0,b_3-1,\ldots,b_{d+1}-1),& b_2=1,\\
-c_d(0,\rho(b_3),\ldots,\rho(b_{d+1})),
 & b_2=0\ \text{and}\ b_j\neq1\ (j\geq3),\\
0,&\text{otherwise}.
\end{cases}
\]
The first line is the $L$-case, the second is the $R$-case, and in the
remaining cases the endpoint contributions either cancel or have no
admissible parent. This is exactly the recursion encoded by $L$ and $R$.
Thus, the two child
classes describe the result after expanding the recursion for the whole
level, collecting equal structure numbers, and canceling the opposite
interior-endpoint contributions; they are not a term-by-term offspring
statement for one individual monomial.
}
\end{itemize}

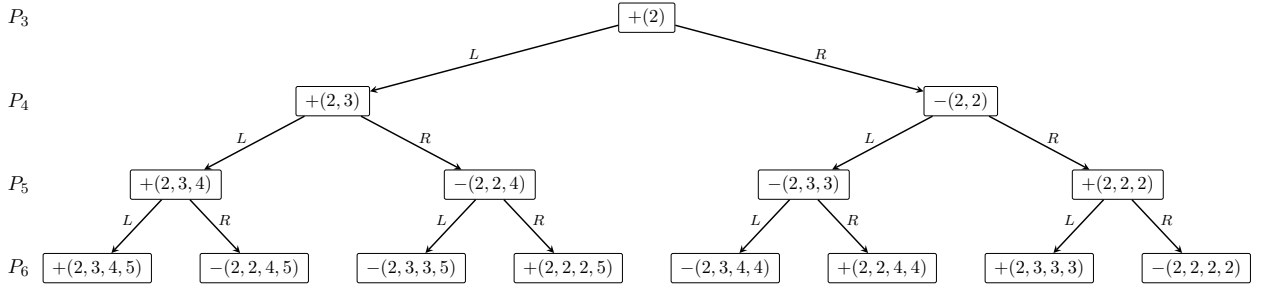
\begin{figure}[htbp]
\centering
\resizebox{\textwidth}{!}{%
\begin{tikzpicture}[
  sn/.style={draw=black,text=black,rounded corners=1pt,inner xsep=5pt,inner ysep=3pt,font=\small},
  ed/.style={->,>=stealth,draw=black,thick},
  el/.style={text=black,font=\scriptsize,fill=white,inner sep=1pt}
]
\node[text=black,font=\bfseries] at (-12,0) {$P_3$};
\node[text=black,font=\bfseries] at (-12,-1.6) {$P_4$};
\node[text=black,font=\bfseries] at (-12,-3.2) {$P_5$};
\node[text=black,font=\bfseries] at (-12,-4.8) {$P_6$};
\node[sn] (p3) at (0,0) {$+(2)$};
\node[sn] (p4l) at (-6,-1.6) {$+(2,3)$};
\node[sn] (p4r) at (6,-1.6) {$-(2,2)$};
\node[sn] (p5ll) at (-9,-3.2) {$+(2,3,4)$};
\node[sn] (p5lr) at (-3,-3.2) {$-(2,2,4)$};
\node[sn] (p5rl) at (3,-3.2) {$-(2,3,3)$};
\node[sn] (p5rr) at (9,-3.2) {$+(2,2,2)$};
\node[sn] (p6a) at (-10.5,-4.8) {$+(2,3,4,5)$};
\node[sn] (p6b) at (-7.5,-4.8) {$-(2,2,4,5)$};
\node[sn] (p6c) at (-4.5,-4.8) {$-(2,3,3,5)$};
\node[sn] (p6d) at (-1.5,-4.8) {$+(2,2,2,5)$};
\node[sn] (p6e) at (1.5,-4.8) {$-(2,3,4,4)$};
\node[sn] (p6f) at (4.5,-4.8) {$+(2,2,4,4)$};
\node[sn] (p6g) at (7.5,-4.8) {$+(2,3,3,3)$};
\node[sn] (p6h) at (10.5,-4.8) {$-(2,2,2,2)$};
\draw[ed] (p3) -- node[el,pos=.55,above left] {$L$} (p4l);
\draw[ed] (p3) -- node[el,pos=.55,above right] {$R$} (p4r);
\draw[ed] (p4l) -- node[el,pos=.55,above left] {$L$} (p5ll);
\draw[ed] (p4l) -- node[el,pos=.55,above right] {$R$} (p5lr);
\draw[ed] (p4r) -- node[el,pos=.55,above left] {$L$} (p5rl);
\draw[ed] (p4r) -- node[el,pos=.55,above right] {$R$} (p5rr);
\draw[ed] (p5ll) -- node[el,pos=.52,above left] {$L$} (p6a);
\draw[ed] (p5ll) -- node[el,pos=.52,above right] {$R$} (p6b);
\draw[ed] (p5lr) -- node[el,pos=.52,above left] {$L$} (p6c);
\draw[ed] (p5lr) -- node[el,pos=.52,above right] {$R$} (p6d);
\draw[ed] (p5rl) -- node[el,pos=.52,above left] {$L$} (p6e);
\draw[ed] (p5rl) -- node[el,pos=.52,above right] {$R$} (p6f);
\draw[ed] (p5rr) -- node[el,pos=.52,above left] {$L$} (p6g);
\draw[ed] (p5rr) -- node[el,pos=.52,above right] {$R$} (p6h);
\end{tikzpicture}%
}
{
\caption{The first four levels of the binary tree of signed structure numbers after expanding the recursion for $P_n$, collecting equal structure numbers, and cancelling opposite terms. The arrows agree with the rules for $L$ and $R$ stated above. The nodes record signed structure classes, not the multiplicities inside their fibers; those multiplicities are included in $h(f^{-1}(\nu))$. In particular, the all-right node at level $P_{k+1}$ is $(-1)^{k-2}(2,\ldots,2)$, and its $h$-evaluation is $(-1)^{k-2}L_2^{k-1}$.}
\label{fig:structure-number-tree}
}
\end{figure}

{Define the evaluation function, for every $d\geq1$, by
\[
h\!\left(\varepsilon\prod_{j=1}^{d}\tau_{-s_j}\right)
:=\varepsilon\prod_{j=1}^{d}s_j,
\]
and extend $h$ linearly.} {For a signed structure number $\nu=\varepsilon(2,a_2,\ldots,a_{n-2})$, define}
{
\[
h\big(f^{-1}(\nu)\big)
:=\sum_{\substack{M\in\mathcal P_n\\ f(M)=\nu}}h(M).
\]
}
{Thus}
{
\[
S_n(\ell_1,\ldots,\ell_n):=h(P_n)
=\sum_{\nu\in f(\mathcal P_n)}h\big(f^{-1}(\nu)\big).
\]
}

{
On the free vector space whose basis consists of the tagged expressions
$h(f^{-1}(\nu))$ (with the level included in the tag), define
\[
\wt L\big(h(f^{-1}(\nu))\big):=h\big(f^{-1}(L\nu)\big),
\qquad
\wt R\big(h(f^{-1}(\nu))\big):=h\big(f^{-1}(R\nu)\big),
\]
and extend $\wt L$ and $\wt R$ linearly. Thus they act on
$S_n(\ell_1,\ldots,\ell_n)$ without ambiguity.
}
Now we prove the following lemma\textcolor{red}{.}
   \begin{lemma}\label{combin}
       For any $n\geq 3$, $S_{n}(\ell_1,{\ldots},\ell_n)=h(P_n)=h(\tau_{-L_2}\cdot\tau_{-L_n}^{n-3})$.
   \end{lemma}

\begin{proof}
{First, it is easy to check that the statement is true for $n=3,4$.}

{Assume that the statement is true for all $3\le n\le k$. For $n=k+1$, observe that}
{
every path in Figure~\ref{fig:structure-number-tree} is either the all-right path or has a unique last left edge followed by $i$ right edges, for some $0\leq i\leq k-3$. Consequently,
\begin{equation}\label{eq: totalexp}
	S_{k+1}(\ell_1,\cdots,\ell_{k+1})=\sum_{i=0}^{k-3}\wt R^{i}\circ \wt L(S_{k-i}(\ell_1,\cdots,\ell_{k-i}))+(-1)^{k-2}L_2^{k-1}.
\end{equation}
}
{By the induction hypothesis, $S_k=h(\tau_{-L_2}\tau_{-L_k}^{k-3})$. The last term in~\eqref{eq: totalexp} depends only on $L_2$.}
	
{
For $0\leq i\leq k-3$, the branch whose last left edge is followed by $i$ right edges satisfies, by the induction hypothesis,
\begin{align*}
\wt R^{i}\circ \wt L(S_{k-i})
=&(-1)^iL_2^{i+1}L_{k+1}^{k-i-3}
\sum_{\substack{I\subseteq\{3,\ldots,k+1\}\\ |I|=i+1}}
\left(L_2+\sum_{j\in I}\ell_j\right).
\end{align*}
Indeed, a chain in this branch has $i+1$ initial empty subsets and then a
fixed subset $I$ of cardinality $i+1$. After removing $I$ from all later
subsets, the remaining fiber is the binary-tree sum at level $P_{k-i}$ on
the complement of $I$, with base length
$L_2+\sum_{j\in I}\ell_j$ and total length $L_{k+1}$. The first $i+1$
factors contribute $L_2^{i+1}$, the $i$ right edges contribute the sign
$(-1)^i$, and the induction hypothesis gives the displayed formula.
Here each subset $I$ occurs exactly once; summing over all permutations would introduce an incorrect factorial multiplicity. Furthermore,
\begin{align*}
\sum_{|I|=i+1}\left(L_2+\sum_{j\in I}\ell_j\right)
=&\binom{k-1}{i+1}L_2+\binom{k-2}{i}(L_{k+1}-L_2),
\end{align*}
which depends only on $L_2$ and $L_{k+1}$.
}

{
Using Pascal's identity, the last displayed sum can be rewritten as
\[
\sum_{|I|=i+1}\left(L_2+\sum_{j\in I}\ell_j\right)
=\binom{k-2}{i}L_{k+1}+\binom{k-2}{i+1}L_2.
\]
Substituting this into~\eqref{eq: totalexp}, the adjacent terms telescope:
\begin{align*}
S_{k+1}
={}&\sum_{i=0}^{k-3}(-1)^iL_2^{i+1}L_{k+1}^{k-i-3}
\left[\binom{k-2}{i}L_{k+1}+\binom{k-2}{i+1}L_2\right]
+(-1)^{k-2}L_2^{k-1}\\
={}&L_2L_{k+1}^{k-2}.
\end{align*}
This completes the induction.
}
\end{proof}
 {
We now complete the proof under $\widehat\P_\lambda$. Let
$Y_0,Y_3,\ldots,Y_n$ be independent random variables having the laws of the
increments of $(\widehat\tau_a)_{0\le a\le L_n}$ over intervals of lengths
$L_2,\ell_3,\ldots,\ell_n$, respectively. For a signed monotone item
\[
M=\varepsilon_M\prod_{j=1}^{n-2}\tau_{-s_{M,j}},
\]
write
\[
s_{M,j}=L_2+\sum_{i\in A_{M,j}}\ell_i,
\qquad
\varnothing=A_{M,1}\subseteq A_{M,2}\subseteq\cdots\subseteq A_{M,n-2}
\subseteq\{3,\ldots,n\}.
\]
Since the sets $A_{M,j}$ are nested, each increment between two successive
levels can be split into independent increments of lengths
$\ell_i$, $i\in A_{M,j}\setminus A_{M,j-1}$. Different monotone items
arrange these deterministic intervals in different orders, but by
stationarity the required equality in law holds termwise, so the same
auxiliary product law for $(Y_0,Y_3,\ldots,Y_n)$ can be used before applying
linearity. Stationarity and independence of the increments therefore give
\[
\widehat\E_\lambda[M]
=\varepsilon_M\E\!\left[
\prod_{j=1}^{n-2}\left(Y_0+\sum_{i\in A_{M,j}}Y_i\right)
\right].
\]
Lemma~\ref{combin} is an identity of formal polynomials in
$L_2,\ell_3,\ldots,\ell_n$. Substituting $L_2=Y_0$ and $\ell_i=Y_i$, summing
over all signed monotone items, and taking expectations gives
\begin{align*}
\widehat\E_\lambda[P_n]
=&\E\!\left[Y_0\left(Y_0+\sum_{i=3}^nY_i\right)^{n-3}\right]=\widehat\E_\lambda[\tau_{-L_2}\tau_{-L_n}^{n-3}].
\end{align*}
This proves the required identity under $\widehat\P_\lambda$, and hence the
proposition. Powers such as $Y_i^r$ already include the terms in which the
same excursion is selected more than once, so the diagonal terms are not
omitted.
}
\end{proof}	
	
\begin{proof}[Proof of Theorem~\ref{thm:area QHn}]
        {Combining Lemma~\ref{lem-QHn-translate} with Proposition~\ref{prop:combinatorics} proves~\eqref{areacompute}. For later use, conditioning the latter expectation on $\cF_{L_n}$ gives}
	\begin{align*}
	    \E^\beta\big[P_n\prod_{i\ge 1}g(x_i) \big]=&\E^{\beta}[\tau_{-L_2}\tau_{-L_n}^{n-3}\prod_{i\geq 1}g(x_i)]\\
	    =& \E^\beta\big[\E_{L_n}\big[\tau_{-L_2}\tau_{-L_n}^{n-3}\big]
	      \prod_{e\in \mathrm{Exc}_{L_n} }F(e)\big].
	\end{align*}
		 {We have $\tau_{-L_n}=\sum_{(e,s)\in \mathrm{Exc}_{L_n}}T(e)$ and $\tau_{-L_2}=\sum_{(e,s)\in \mathrm{Exc}_{L_n}}T(e)1_{\{s\le L_2\}}$. Writing $n-2=\sum_{i=1}^{n-2}ik_i$, we obtain}
    \begin{align*}
        \E_{L_n}\big[\tau_{-L_2}\tau_{-L_n}^{n-3}\big]= \frac{L_2}{L_n} {\sum_{\substack{k_1,\ldots,k_{n-2}\ge0\\\sum_{i=1}^{n-2}ik_i=n-2}}} A((k_i)_i)\times {\Phi[(k_i)_i,T]},
    \end{align*}
{where}
    \begin{align*}
        {A((k_i)_i):=\frac{(n-2)!}{\prod_{i=1}^{n-2}k_i!(i!)^{k_i}},} \qquad
{\Phi[(k_i)_i,T]:=}\sum_{\substack{
		      e_{{i,j}} \in \mathrm{Exc}_{L_n}\\
		    {e_{{i,j}} \ne e_{{i',j'}}\ \text{whenever }}
		    (i,j)\ne (i',j')\\
		    1\le i,i'\le n-2,\ 1\le j\le k_i,\ 1\le j'\le k_{i'}}
		    }
		    \prod_{i=1}^{n-2} \prod_{j=1}^{k_i}  {T(e_{{i,j}})}^{i}\textcolor{red}{.}
    \end{align*}

By Palm's theorem for Poisson point processes (see, e.g., \cite[Page 5]{kallenberg-random-measures}), we have

\begin{align}\label{midstep}
 &\E^\beta \big[\E_{L_n}\big[\tau_{-L_2}\tau_{-L_n}^{n-3}\big]\prod_{e\in \mathrm{Exc}_{L_n} }F(e)\big] \\
 =& \frac{L_2}{L_n} \cdot \E^{\beta}\big[ \prod_{e\in \mathrm{Exc}_{L_n} }F(e)\big] \cdot {\sum_{\substack{k_1,\ldots,k_{n-2}\ge0\\\sum_{i=1}^{n-2}ik_i=n-2}}} A((k_i)_i) \cdot {L_n}^{\sum k_i} \cdot \prod_{i=1}^{n-2} {\big[ \int {T(e)}^{i} F(e) \ul N' (d e)\big]}^{k_i}.
\end{align}

Note that 
\begin{align}\label{mid2step}
\E^{\beta}\big[ \prod_{e\in \mathrm{Exc}_{L_n} }F(e)\big] = \E^{\beta}\big[\prod_{i\ge 1} g(x_i)\big] = e^{-L_n \cdot \sqrt{\mu/{\sin(\pi{\gamma}^2/4)}}},
\end{align}
by the L\'evy-process identity {\cite[Proposition 6.7]{Int-CLE}}.

{These formulas will be combined below with Proposition~\ref{thm-QA2-area} to obtain the derivative form of the answer.}

\end{proof}

{For computational convenience, we now use the stable L\'evy process $\zeta$ with L\'evy measure $\Pi(dx)=\Gamma(-\beta)^{-1}1_{\{x>0\}}x^{-\beta-1}\,dx$. Writing $G_\beta=\Gamma(-\beta)$, we have $(\zeta_{G_\beta t})_{t\ge0}\stackrel d=(\zeta'_t)_{t\ge0}$. Let $N_\ell$ be the probability law of a $\zeta$-excursion of duration $\ell$. Its excursion measure is}

\begin{align}\label{excursion1}
	\ul N = \frac{1}{\beta\cdot \Gamma(1-1/\beta)} \cdot  \int_0^\infty d\ell\,\ell^{-1-1/\beta}\,N_\ell\;{.}
\end{align}
{In accordance with the time change above, $N'_{\ell/G_\beta}$, not $N'_{G_\beta\ell}$, corresponds to $N_\ell$. If $e'$ and $e$ are corresponding excursions, then $T(e')=T(e)/G_\beta$ and their jump multisets agree. Hence, for every $k\ge1$,}
\[
{\int T(e')^kF(e')\,\ul N'(de')
=G_\beta^{-k}\int T(e)^kF(e)\,\ul N(de).}
\]
{Before proceeding to Section~\ref{string equation},}
we introduce a function $\Bar{u}(s)$ corresponding to the excursion measure $\ul N $:
\begin{align}
    \Bar{u}(s)=\int (e^{\frac{T(e)}{\Gamma(-\beta)}s}-1) \prod_{x\in J_e} g(x) \ul N (d e)-\sqrt{\frac{\mu}{\sin(\frac{\pi\gamma^2}{4})}}.
\end{align}
{Thus $\Bar{u}(0)=-\sqrt{\frac{\mu}{\sin(\frac{\pi\gamma^2}{4})}}$ and, for $k\ge1$,
$\Bar{u}^{(k)}(0)={\left(\frac{1}{\Gamma(-\beta)}\right)}^{k}
\int T(e)^k \prod_{x\in J_e} g(x) \ul N (d e)$.}

\begin{lemma}
For $\beta=\frac{4}{\gamma^2}+\frac{1}{2}$, {$\mu>0$, and every integer $k\ge0$},
\begin{align}
    \Bar{u}^{(k)}(0)=C^{(k)}_{\gamma}\times \mu^{\frac{1-\beta k}{2}},
\end{align}  
where $C^{(k)}_\gamma$ {depends only} on $\gamma$.
\end{lemma}

\begin{proof}
{The case $k=0$ follows directly from the definition. For $k\ge1$, substitute~\eqref{excursion1} in the formula for $\Bar u^{(k)}(0)$ and use the scaling of $N_\ell$. Writing the result as a double integral and making the change of variables $r=\mu^{\beta/2}\ell$ factors out $\mu^{(1-\beta k)/2}$; the remaining integral depends only on $\gamma$. It is finite: apply~\eqref{midstep} with $n=k+2$. The expectation there is finite by Theorem~\ref{thm:area QHn}; all terms in the Bell sum are non-negative, and the term $k_k=1$ contributes $L_n\int T(e)^kF(e)\,\ul N'(de)$. The time-change identity above therefore gives the corresponding finiteness for $\ul N$. Denote the resulting finite constant by $C_\gamma^{(k)}$.}
\end{proof}

{Using $\Bar u$, we obtain another form of the main result of this subsection.}
\begin{theorem}\label{area-ufunc}
{{By~\eqref{areacompute}, the area transform can be written in terms of the derivatives of $\Bar u$ at $0$:}
\begin{align}
    \QH_n(\ell_1,\ell_2,\ldots,\ell_n)[e^{-\mu A}]=&\frac{\cos(\pi (\frac4{\gamma^2}-1))}{\pi(\bar{R}(\gamma ; 1,1))^{n-2}} \cdot \frac{1}{\sqrt{\prod_{k=1}^n\ell_k}}\cdot \left.\partial_s^{n-3}\left(\Bar{u}'(s)e^{\Bar{u}(s)L_n}\right)\right|_{s=0}\\
    =&\frac{\cos(\pi (\frac4{\gamma^2}-1))}{\pi(\bar{R}(\gamma ; 1,1))^{n-2}}\cdot\frac{1}{\sqrt{\prod_{k=1}^n\ell_k}} \cdot e^{-\sqrt{\frac{\mu}{\sin(\frac{\pi\gamma^2}{4})}}\cdot L_n}\cdot\left(\sum_{j=0}^{n-3}A_j L_n^j\right).\label{areacompute2}
\end{align}
{Here}
\begin{align}
{A_j=\sum_{\substack{k_1,\ldots,k_{n-2}\ge0\\
    \sum_{i=1}^{n-2}ik_i=n-2\\
    \sum_{i=1}^{n-2}k_i=j+1}}
    \frac{(n-2)!}{\prod_{i=1}^{n-2}k_i!(i!)^{k_i}}
    \prod_{i=1}^{n-2}\bigl(\Bar u^{(i)}(0)\bigr)^{k_i}.}
\end{align}
}
\end{theorem}

\begin{proof}
{The time-change identity above gives $\Bar u^{(k)}(0)=\int T(e)^kF(e)\,\ul N'(de)$ for $k\ge1$. Moreover,
    \[
    \partial_s^{n-2}e^{L_n\Bar u(s)}
    =L_n\partial_s^{n-3}\!\left(\Bar u'(s)e^{L_n\Bar u(s)}\right).
    \]
    Combining this identity with~\eqref{midstep}, \eqref{mid2step}, and Proposition~\ref{thm-QA2-area} gives the first equality. The complete Bell-polynomial expansion gives the second equality and the displayed formula for $A_j$.}
\end{proof}
To compute $\QH_n(\ell_1,\ell_2,\ldots,\ell_n)[e^{-\mu A}]$, it suffices to determine the derivatives of $\Bar{u}$ at the origin. A direct evaluation of the corresponding excursion integrals does not appear to be available. In the next section, we instead derive a string equation for $\Bar{u}$ and use Lagrange inversion to determine these derivatives, equivalently the excursion integrals of all orders.

We next define $u(s)$ initially for $s\le 0$. The string equation derived below extends $u$ real-analytically to a neighborhood of the origin. All derivatives at $0$ are understood with respect to this continuation, and
$
u^{(k)}(0)=\Bar u^{(k)}(0), k\ge 0.
$

\begin{definition}
    For $s\leq 0$, we define 
    $$
    u(s):=\int \Big(e^{\frac{T(e)}{\Gamma(-\beta)}s} \prod_{x\in J_e} g(x)-1\Big) \ul N (d e)=\int \Big(e^{T(e)s} \prod_{x\in J_e} g(x)-1\Big) \ul N' (d e).
    $$
\end{definition}
{For $k\ge1$, differentiating for $s<0$ and then letting $s\uparrow0$ gives, by monotone convergence,
$u^{(k)}(0)=\Gamma(-\beta)^{-k}\int T(e)^k\prod_{x\in J_e}g(x)\,\ul N(de)$.
The analytic continuation obtained below has these same Taylor coefficients.  Thus the only remaining nontrivial point is to verify $u(0)=-\sqrt{\mu/\sin(\gamma^2\pi/4)}$; this is proved in Proposition~\ref{same of u}.}


\section{{Campbell's theorem and the string equation for $u(s)$}}\label{string equation}

{
It is predicted in physics that matrix models and Liouville conformal field theory give equivalent descriptions of two-dimensional quantum gravity. Formally, a matrix model provides a discrete description of maps coupled to a statistical-mechanics model, and its double-scaling limit is expected to describe a decorated random planar map converging to LQG coupled with CLE. The string equation {is predicted to determine} this double-scaling limit. In this section, we show that the function $u(s)$ defined above satisfies the corresponding string equation and thereby determines the genus-zero area transform.

We now carry out the required stable-L\'evy-excursion computations. The string equation can be viewed as a variant of Campbell's theorem, which we first recall.
}

\begin{lemma}[Campbell’s theorem]
{
  {Let $f:\mathbb X\to[0,\infty]$ be measurable and $s\ge0$. The Laplace functional of a Poisson point process $\Phi$ on $\mathbb X$ with intensity $\Lambda$ is}
$$
\mathbb{E}\left[e^{-s\sum_{{x\in\Phi}} f({x}) }\right]=e^{-\int_{\mathbb{X}}\left(1-e^{-sf(x)}\right) \Lambda(\mathrm{d} x)} .
$$
}
 
\end{lemma}

{
Suppose $
\lexp =\frac{4}{\gamma^2}+\frac12 \in (\frac32,2)$. Recall the L\'evy process {$(\zeta_t)_{t\ge 0}$} with L\'evy measure $\Pi(dx)=\frac{1}{\Gamma(-\beta)}1_{\{x>0\}}x^{-\beta-1}\,dx$, {introduced above}. Let ${N}_1$ be the probability measure on excursions of duration 1 and ${N}_\ell$ the law of $(\ell^{1/\beta}b_{t/\ell}:0\le t\le\ell)$ for $b\sim N_1$. {Let $\Delta\zeta_t:=\zeta_t-\zeta_{t-}$ denote the jump at time $t$, and set $\tau_{-1}:=\inf\{t\ge0:\zeta_t=-1\}$.}
}

Now, let us demonstrate some calculations for explicit functionals related to L\'evy excursions before proceeding further.

\begin{proposition}[{\cite[Lemma 4.2]{nested-stat}}]\label{lem:palm-levy-excursion}
	For $\beta=\frac{4}{\gamma^2}+\frac{1}{2} \in (\frac{3}{2},2)$, let $G\colon [0,\infty)\to [0,\infty)$ be twice continuously differentiable with $G(0)=G'(0)=0$. Suppose that for all $\lambda>0$ there exists $\rho(\lambda)>0$ such that
	\begin{align}
		\label{eq:palm-key-statement}
		\lambda = \frac{1}{\Gamma(-\beta)} \int_0^\infty \frac{dh}{h^{1+\beta}}\,(e^{-\rho(\lambda) h - G(h)}-1+\rho(\lambda) h) \;.
	\end{align}
	Note that the integral above is finite for any $\rho(\lambda)>0$ because of the assumption $G(0)=G'(0)=0$. Then
	\begin{align}
		\label{eq44}
		\begin{split}
		-\rho(\lambda) &= \log \E\left( e^{-\lambda {\tau_{-1}}-\sum_{t< {\tau_{-1}}} G({\Delta\zeta_t})} \right) \\
		&= \frac{1}{\beta \Gamma(1-1/\beta)} \int_0^\infty \frac{d\ell}{\ell^{1+1/\beta}}\left( e^{-\lambda \ell}\,\E\left( e^{-\sum_{t\le 1} G(\ell^{1/\beta}\Delta b_t) } \right) - 1\right)\;.
		\end{split}
	\end{align}
\end{proposition}

Now we recall the function defined in Theorem{~}\ref{thm:area QHn}. To simplify notation, we set $M=\sqrt{\frac{\mu}{\sin(\frac{\gamma^2}{4}\pi)}}$ {and obtain}
\begin{align*}
		g(x) &= \QD^{\#}(1)[e^{-x^2\mu A}]= \frac{2}{\Gamma(\frac{4}{\gamma^2})}\,(\frac{Mx}{2})^{\frac{4}{\gamma^2}}\,K_{\frac{4}{\gamma^2}}(Mx),\quad\text{for $x>0$}{.}
\end{align*}
Note that $g(0)=1$. {Define $G(x)=-\log g(x)$. By the FZZ formula, $G$ is twice continuously differentiable and $G(0)=G'(0)=0$. Moreover, the right-hand side of~\eqref{eq:palm-key-statement}, viewed as a function of $\rho\ge0$, is continuous and strictly increasing, is non-positive at $\rho=0$, and tends to $+\infty$ as $\rho\to\infty$. Thus the required $\rho(\lambda)>0$ exists uniquely for every $\lambda>0$, so Proposition~\ref{lem:palm-levy-excursion} applies and~\eqref{eq44} gives}
\begin{align}
    -\rho(\lambda)&=\frac{1}{\beta \Gamma(1-1/\beta)} \int_0^\infty \frac{d\ell}{\ell^{1+1/\beta}}\left( e^{-\lambda \ell}\,{\E}\left(\prod_{t<1} g(\ell^{1/\beta}\Delta b_t)) \right) - 1\right)\\ 
    &= \int (e^{-T(e)\lambda}\prod_{x\in J_e} g(x)-1) \ul N (d e).\label{equ10}
\end{align}

\begin{proposition}\label{same of u}
{For $\lambda>0$, set $s(\lambda):=\rho(\lambda/\Gamma(-\beta))$. Then $s(\lambda)=-u(-\lambda)$, and $s$ extends continuously to $0$ with
$s(0)=\sqrt{\mu/\sin(\gamma^2\pi/4)}$. Consequently,
$u(0)=-\sqrt{\mu/\sin(\gamma^2\pi/4)}$.}
\end{proposition}

\begin{proof}
{For $\lambda>0$, equation~\eqref{equ10} gives $s(\lambda)=-u(-\lambda)$. Letting $\lambda\downarrow0$ in~\eqref{eq44}, using monotone convergence, and then applying the same argument as for its second equality gives}
 
\begin{align*}
    s(0)&= -\frac{1}{\beta \Gamma(1-1/\beta)} \int_0^\infty \frac{d\ell}{\ell^{1+1/\beta}}\left( \E\left(\prod_{{t\le1}} g(\ell^{1/\beta}\Delta b_t)\right) - 1\right)\\
    &=-\log \E\left(\prod_{t<\tau_{-1}} g({\Delta\zeta_t})\right)
    =\sqrt{\frac{\mu}{\sin(\frac{\gamma^2}{4}\pi)}}{.}
\end{align*}
{The third equality follows from the L\'evy-process identity in \cite[Proposition 6.7]{Int-CLE}.}
\end{proof}

{On every common domain where the two excursion integrals are finite, the functions $u$ and $\Bar u$ agree. Indeed, with $F(e)=\prod_{x\in J_e}g(x)$,}
\[
{u(r)-\Bar u(r)=\int(F(e)-1)\,\ul N(de)+M=u(0)+M=0.}
\]

By the definition of $s(\lambda)$, equation \eqref{eq:palm-key-statement} can be written as
\begin{align}\label{equ11}
    \lambda = \int_0^\infty \frac{dh}{h^{1+\beta}}\,(e^{-s(\lambda) h}\cdot g(h)-1+s(\lambda) h){.}
\end{align}


To simplify the computation{,} we introduce another parameter $\alpha$. For $\Re(\alpha)>\frac{4}{\gamma^2}$, we {define}
\begin{align}\label{analyticcontin}
    \lambda(\alpha){:=} \int_0^\infty \frac{dh}{h^{1+\beta}}\,e^{-s(\lambda) h}\cdot g(h)\cdot h^{\frac{1}{2}+\alpha}{.}
\end{align}
{For brevity, write $s=s(\lambda)$.  By the integral formula in Proposition~\ref{besselint} and Lemma~\ref{eulertrans}, for $\Re(\alpha)>4/\gamma^2$,}
{
\begin{align*}
&\int_0^{\infty} \ell^{\alpha-1} e^{-s \ell}
K_{\frac{4}{\gamma^2}}(M \ell)\,d\ell \\
&\quad=\sqrt{\frac{\pi}{2 M}}(M+s)^{1 / 2-\alpha}
\frac{\Gamma(\alpha+\frac{4}{\gamma^2})
\Gamma(\alpha-\frac{4}{\gamma^2})}{\Gamma(\alpha+1 / 2)}
{}_2F_1\!\left(\frac12+\frac{4}{\gamma^2},
\frac12-\frac{4}{\gamma^2};\frac12+\alpha;
\frac{M-s}{2M}\right).
\end{align*}
Consequently,
\begin{align}\label{hypergeo1}
\lambda(\alpha)
={}&\frac{2^{1-\frac{4}{\gamma^2}}M^{\frac{4}{\gamma^2}}}
{\Gamma(\frac{4}{\gamma^2})}\sqrt{\frac{\pi}{2M}}
(M+s)^{1/2-\alpha}\nonumber\\
&\quad\times
\frac{\Gamma(\alpha+\frac{4}{\gamma^2})
\Gamma(\alpha-\frac{4}{\gamma^2})}{\Gamma(\alpha+1/2)}
{}_2F_1\!\left(\frac12+\frac{4}{\gamma^2},
\frac12-\frac{4}{\gamma^2};\frac12+\alpha;
\frac{M-s}{2M}\right).
\end{align}
}

\begin{proposition}[{\cite[Volume 2, Section 7.7. (26)]{Bateman1953HigherTF}}]\label{besselint}
For $\operatorname{Re} \mu>|\operatorname{Re} \nu|$, {$\beta>0$}, and $\operatorname{Re}(\alpha+\beta)>0$, {with complex powers on their principal branches},
\begin{align}\label{bessel int formula}
\int_0^{\infty} x^{\mu-1} e^{-\alpha x} K_\nu(\beta x) d x=\frac{\sqrt{\pi}(2 \beta)^\nu}{(\alpha+\beta)^{\mu+\nu}} \frac{\Gamma(\mu+\nu) \Gamma(\mu-\nu)}{\Gamma\left(\mu+\frac{1}{2}\right)} {}_{2}F_{1}\left(\mu+\nu, \nu+\frac{1}{2} ; \mu+\frac{1}{2} ; \frac{\alpha-\beta}{\alpha+\beta}\right){.}
\end{align}
\end{proposition}
\begin{proof}
    See Appendix.
\end{proof}
\begin{lemma}[{Pfaff's hypergeometric transformation}]\label{eulertrans}
   {For $c\notin\{0,-1,-2,\ldots\}$,}
    $${}_{2}F_{1}(a,b;c;z)=(1-z)^{-a}{}_{2}F_{1}\left(a,c-b;c;{\frac{z}{z-1}}\right).$$
   {This identity holds initially for compatible principal branches and elsewhere by analytic continuation away from the standard hypergeometric branch cut.}
\end{lemma}
{We now justify the continuation to $\alpha=-\frac12$. Put $q=4/\gamma^2$, $\beta=q+\frac12$, and $f_s(h)=e^{-s(\lambda)h}g(h)$. Since $g(0)=1$ and $g'(0)=0$,}
\[
{f_s(h)=1-s(\lambda)h+O(h^2)\qquad(h\downarrow0).}
\]
{Holding $s(\lambda)$ fixed during the continuation, split off these two Taylor terms on $(0,1)$. For $\Re\alpha>q$,}
\begin{align}
\lambda(\alpha)={}&{\int_0^\infty h^{\alpha-q-1}
\left(f_s(h)-\mathbf1_{\{h<1\}}(1-s(\lambda)h)\right)\,dh
+\frac1{\alpha-q}-\frac{s(\lambda)}{\alpha-q+1}.}\label{mellin-regularization}
\end{align}
{At $\alpha=-\frac12$, the two rational terms equal}
\[
{-\frac1\beta+\frac{s(\lambda)}{\beta-1}
=\int_1^\infty h^{-1-\beta}(-1+s(\lambda)h)\,dh.}
\]
{The remaining integral is analytic there: its integrand is $O(h^{1/2-q})$ at zero, which is integrable because $q<3/2$, and it decays exponentially at infinity. Consequently,}
\begin{align}
{\lambda(-\tfrac12)=\int_0^\infty h^{-1-\beta}
\left(e^{-s(\lambda)h}g(h)-1+s(\lambda)h\right)\,dh,}\label{mellin-finite-part}
\end{align}
{which is exactly~\eqref{equ11}, with no additional finite-part constant.}

\begin{proposition}[Hypergeometric Equation]\label{hypergeo}
{Put $w=(M-s(\lambda))/(2M)$.  For $|w|<1$, the series identity below holds; its final hypergeometric expression gives the analytic continuation along the negative real $w$-axis:}
\begin{align}
    \lambda &=\frac{2^{1-\frac{4}{\gamma^2}} M^{\frac{4}{\gamma^2}}}{\Gamma(\frac{4}{\gamma^2})}\sqrt{\frac{\pi}{2 M}} (M+s(\lambda))\cdot \frac{1}{(\frac{1}{4}-({\frac{4}{\gamma^2}})^{2})}\sum_{k=1}^{\infty} \frac{\Gamma(1 / 2+\frac{4}{\gamma^2}+k) \Gamma(1 / 2-\frac{4}{\gamma^2}+k)}{\Gamma(k)\Gamma(k+1) }\left(\frac{M-s(\lambda)}{2 M}\right)^k\nonumber\\
    &={\frac{2^{1-{\frac{4}{\gamma^2}}}}{\Gamma({\frac{4}{\gamma^2}})}\sqrt{2\pi}M^{{\frac{4}{\gamma^2}}+\frac{1}{2}} \frac{\pi}{\cos{(\frac{4}{\gamma^2}\pi)}} \frac{M-s(\lambda)}{2M} {}_{2}{F}_{1}\left(\frac{1}{2}-{\frac{4}{\gamma^2}},\frac{1}{2}+{\frac{4}{\gamma^2}};2;\frac{M-s(\lambda)}{2M}\right).}
\end{align}
\begin{proof}
  {Set $q=4/\gamma^2$, $A=\frac12+q$, $B=\frac12-q$, and $w=(M-s(\lambda))/(2M)$. The limit $\alpha\to-\frac12$ in~\eqref{hypergeo1}, justified by~\eqref{mellin-finite-part}, gives the first line coefficientwise. For $|w|<1$, reindexing the series and using $M+s(\lambda)=2M(1-w)$ gives
    \[
    \frac{M+s(\lambda)}{AB}\sum_{k\ge1}
    \frac{\Gamma(A+k)\Gamma(B+k)}{\Gamma(k)\Gamma(k+1)}w^k
    =2M(1-w)\Gamma(A)\Gamma(B)w\,{}_2F_1(A+1,B+1;2;w).
    \]
    Since $A+B=1$, Euler's transformation yields
    $(1-w){}_2F_1(A+1,B+1;2;w)={}_2F_1(A,B;2;w)$, and the reflection formula gives $\Gamma(A)\Gamma(B)=\pi/\cos(\pi q)$. This proves the second line with no extra minus sign. The series representation is valid for $|w|<1$; the final hypergeometric expression supplies its analytic continuation along the negative real axis.}
\end{proof}
\end{proposition}

{The above proposition shows that $s(\lambda)$ satisfies the hypergeometric equation. Define}
\begin{align}
{L_\mu(\gamma):=\frac{2^{1-{\frac{4}{\gamma^2}}}}{\Gamma({\frac{4}{\gamma^2}})}\sqrt{2\pi}M^{{\frac{4}{\gamma^2}}+\frac{1}{2}} \frac{\pi}{\cos(\frac{4\pi}{\gamma^2})}\frac{\gamma^4\pi^2}{8}.}\label{eq:L-mu}
\end{align}
{Here $\mu>0$ is fixed. Since $4/\gamma^2\in(1,3/2)$, $\cos(4\pi/\gamma^2)<0$ and hence $L_\mu(\gamma)<0$. The implicit relation in Proposition~\ref{hypergeo} has nonzero derivative with respect to $s$ at $(\lambda,s)=(0,M)$; henceforth $s$ denotes its unique real-analytic continuation to a neighborhood of $0$, with $s(0)=M$.}

\begin{theorem}[String Equation]\label{thm:string equ}
{Define $u_{\rm st}(\lambda):=-\frac{16}{\gamma^4\pi^2}\frac{M-s(\lambda L_\mu(\gamma))}{2M}$, using the preceding analytic continuation for $|\lambda|$ small. Then the equation in Proposition~\ref{hypergeo} becomes}
\begin{align}\label{simplehyp}
    -\lambda=\frac{u_{\rm st}(\lambda)}{2} {}_{2}{F}_{1}(\frac{1}{2}-{\frac{4}{\gamma^2}},\frac{1}{2}+{\frac{4}{\gamma^2}},2;-\frac{\gamma^4\pi^2}{16}u_{\rm st}(\lambda)){.}
\end{align}
\end{theorem}

\begin{proposition}\label{allderivative}
 Let \(c=\gamma^4\pi^2/8\). Then, for every $k\ge1$, the derivative 
 \begin{align}\label{exact-derivative-transform}
 u^{(k)}(0)=\Gamma(-\beta)^{-k}\int T(e)^k\prod_{x\in J_e}g(x)\,\ul N(de)
 =-Mc\left(-\frac1{L_\mu(\gamma)}\right)^k u_{\rm st}^{(k)}(0).
 \end{align}
 is given explicitly by~\eqref{lagrange-derivative-formula}.
\end{proposition}

\begin{proof}
{By~\eqref{exact-derivative-transform}, it is enough to compute $u_{\rm st}^{(k)}(0)$ from~\eqref{simplehyp}.}
\begin{lemma}[Lagrange Inversion Theorem]
Suppose $z$ is defined as a function of $w$ of the form
$
z=f(w){,}
$
where $f$ is analytic at a point $a$ and $f^{\prime}(a) \neq 0$. Then it is possible to invert or solve the equation for $w$, expressing it in the form $w=g(z)$ given by a power series:
$$
g(z)=a+\sum_{n=1}^{\infty} g_n \frac{(z-f(a))^n}{n !}{,}
$$
where
$$
g_n=\lim _{w \rightarrow a} \frac{d^{n-1}}{d w^{n-1}}\left[\left(\frac{w-a}{f(w)-f(a)}\right)^n\right]{.}
$$\end{lemma}
{Define $\Phi(v):=-\frac v2\,{}_2F_1(\frac12-\frac4{\gamma^2},\frac12+\frac4{\gamma^2};2;-\frac{\gamma^4\pi^2}{16}v)$. Then~\eqref{simplehyp} is $\lambda=\Phi(v)$ with $v=u_{\rm st}(\lambda)$. Since $\Phi(0)=0$ and $\Phi'(0)=-\frac12\ne0$, Lagrange inversion gives, for every $k\ge1$,}
\begin{align}
{u_{\rm st}^{(k)}(0)
=\left.\frac{d^{k-1}}{dv^{k-1}}
\left[-\frac2{{}_2F_1(\frac12-\frac4{\gamma^2},\frac12+\frac4{\gamma^2};2;-\frac{\gamma^4\pi^2}{16}v)}\right]^k\right|_{v=0}.}\label{lagrange-derivative-formula}
\end{align}
\end{proof}

\section{LQG and $\gamma${-deformed Weil--Petersson volumes}}

{In this section, we relate $\gamma$-LQG to Weil--Petersson volumes using the area distribution of the quantum $n$-hole sphere. We first derive the joint law of an $n$-point quantum sphere coupled with CLE and then, subject to Proposition~\ref{reweightwelding}, define a $\gamma$-deformed Weil--Petersson volume whose suitably scaled algebraic continuation tends to the usual Weil--Petersson volume as $\gamma\to0$.}

{For $\kappa\in(\frac{8}{3},4)$, Theorem~\ref{thm: welding QSn} gives the following identity of measures on decorated quantum surfaces:}

\begin{align}
{\mathfrak C_n} =
 C\int_{{(0,\infty)^n}}\mathrm{Weld}(\QH_n(\ell_1,\ell_2,{\ldots},\ell_n),\QD_{1,0}(\ell_1),{\ldots},\QD_{1,0}(\ell_n))\prod_{j=1}^n\ell_j d\ell_j.
\end{align}
{Here $C$ is the constant depending only on $\gamma$ given in Proposition~\ref{disk-welding}.}

{Recall from Definition~\ref{def-QS-2} that a Liouville field describes the three-point quantum sphere $\QS_3$.}
 The corresponding $n$-point formula is as follows.
\begin{proposition}
  {For $n\ge3$, when $(\phi, z_4,\ldots,z_n)$ is sampled from $\int_{{\C}^{n-3}} \LF_\C^{(\gamma, z_i)_{i=1}^n} \prod_{i=4}^n d^2z_i$, the law of the decorated quantum surface $(\C, \phi,z_1,z_2,z_3,\ldots,z_n)/{\sim_\gamma}$ is
$\frac{2(Q-\gamma)^2}{\pi \gamma}\QS_n$.}
\end{proposition}
\begin{proof}
  {Use the rigorous area-insertion identity
   \[\mu_\phi(dz)\,\LF_{\C}^{(\alpha_j,z_j)_j}(d\phi)
   =\LF_{\C}^{(\gamma,z),(\alpha_j,z_j)_j}(d\phi)\,d^2z,
   \]
   obtained by regularization and Girsanov's theorem; see \cite{AHS-SLE-integrability}. Iterating it $n-3$ times biases $\QS_3$ by $\mu_\phi(\C)^{n-3}$ and then samples $z_4,\ldots,z_n$ independently from normalized quantum area. This is precisely the definition of $\QS_n$, and the multiplicative factor is unchanged.}
\end{proof}

A corresponding welding formula can be derived by coupling the
Liouville field with the conformal radii of the CLE loops.
Let $\mathsf m_n(d\boldsymbol\eta\mid\boldsymbol z)$ denote the conditional kernel~\ref{def:mn} of the mutually outermost CLE loops associated with the prescribed points, restricted to the contractible configuration. Define the reweighted measure
\begin{align}
\mathsf{m}^{{(\alpha_i)}_{i}}_n({d\boldsymbol\eta\mid\boldsymbol z}):=\prod_{i=1}^n\left(\frac{1}{2} \mathrm{CR}\left(\eta_i, z_i\right)\right)^{2\Delta_{\alpha_i}-2}\mathsf{m}_n({d\boldsymbol\eta\mid\boldsymbol z}){.}
\end{align}
{Here $\Delta_\alpha:=\frac\alpha2(Q-\frac\alpha2)$.}

{
\begin{proposition}\label{reweightwelding}
Let $n\ge3$ and let $\alpha_i\in(Q-\frac\gamma4,Q)$ for $1\le i\le n$.
\begin{equation}\label{eq:reweighted-welding}
\begin{split}
&\int_{\C^{n-3}}\LF_\C^{(\alpha_i,z_i)_{i=1}^n}
 \otimes\mathsf m_n^{(\alpha_i)_{i=1}^n}
 \prod_{i=4}^n d^2z_i\\
&\quad=C_0\frac{2(Q-\gamma)^2}{\pi\gamma}
\int_{(0,\infty)^n}
\operatorname{Weld}\!\left(
\prod_{i=1}^n\mathcal M_1^{\rm disk}(\alpha_i;\ell_i),
\QH_n(\ell_1,\ldots,\ell_n)\right)
\prod_{i=1}^n\ell_i\,d\ell_i,
\end{split}
\end{equation}
where
\[
C_0=C\left(\frac{\gamma}{2\pi(Q-\gamma)^2}\right)^n{,}
\]
and $C$ is the constant in Proposition~\ref{disk-welding}.
\end{proposition}

\begin{proof}
The proof is the same as~\cite[Theorem 8.7]{Int-CLE}. We omit it here.
\end{proof}
we then take
the area Laplace transform on both sides of~\eqref{eq:reweighted-welding}.

\begin{proposition}\label{prop:gamma-wp}
Let $n\ge3$ and $\alpha_i\in(Q-\frac\gamma4,Q)$ for $1\le i\le n$.
Set
\[
\nu_i=\frac2\gamma(Q-\alpha_i),\qquad
\lambda_i=-\frac{2i}{\gamma}(Q-\alpha_i),
\]
so that $\nu_i=i\lambda_i$, and put
$A_\gamma=\gamma^4\pi^2/8$.  Define
\begin{align}\label{eq:def-gamma-wp}
&\int_{\C^{n-3}}\LF_\C^{(\alpha_i,z_i)_{i=1}^n}[e^{-\mu A}]
\,{\big|\mathsf m_n^{(\alpha_i)_{i=1}^n}(\cdot\mid\boldsymbol z)\big|}
\prod_{i=4}^n d^2z_i\nonumber\\
&\qquad=C_{\rm WP}(\gamma,\boldsymbol\alpha,\mu,n)
\prod_{i=1}^n\frac1{\cosh(\pi\lambda_i)}
V_{0,n}^{\rm WP}(\boldsymbol\lambda),
\end{align}
where
\begin{align}\label{cwp}
C_{\rm WP}={}&C_0\frac{2(Q-\gamma)^2}{\pi\gamma}C_2(\gamma,n)
\left(\prod_{i=1}^nC_1(\alpha_i,\gamma)\right)
\left(\frac M2\right)^{\sum_i\nu_i}M^{1-n/2}\nonumber\\
&\quad\times(-1)^{n-2}\frac{\gamma^4\pi^2}{4}
\left(\frac{(2\pi)^{3/2}}4\right)^n
L_\mu(\gamma)^{-(n-2)},
\end{align}
with
\[
C_1(\alpha,\gamma)=\frac2\gamma\textcolor{red}{\,}2^{-\alpha^2/2}\ol U(\alpha)
\frac2{\Gamma(\frac2\gamma(Q-\alpha))},\qquad
C_2(\gamma,n)=\frac{\cos(\pi(\frac4{\gamma^2}-1))}
{\pi(\bar R(\gamma;1,1))^{n-2}}.
\]
Then
\begin{equation}\label{eq:corrected-gamma-wp}
V_{0,n}^{\rm WP}(\boldsymbol\lambda)
=-\frac12\left.\partial_x^{n-3}\left[
u_{\rm st}'(x)\prod_{i=1}^n
P_{-1/2-i\lambda_i}\bigl(1+A_\gamma u_{\rm st}(x)\bigr)
\right]\right|_{x=0}.
\end{equation}
The Bessel-transform identity used in the proof is first valid for real $\lambda_i$ and then extends
meromorphically to $|\operatorname{Im}\lambda_i|<\frac12$; in particular,
it applies to the displayed physical values.  We use the branch of
$P_\nu$ analytic at $1$ and normalized by $P_\nu(1)=1$.
\end{proposition}

\begin{proof}
Proposition~\ref{reweightwelding}, Theorem~\ref{thm-FZZ}, and
Theorem~\ref{area-ufunc} give the left-hand side of
\eqref{eq:def-gamma-wp} as
\begin{align}
&C_0\frac{2(Q-\gamma)^2}{\pi\gamma}C_2(\gamma,n)
\left(\prod_iC_1(\alpha_i,\gamma)\right)
\left(\frac M2\right)^{\sum_i\nu_i}M^{1-n/2}\,\mathcal J,
\end{align}
where, after setting $x_i=M\ell_i$ and $\widetilde u={u}/M$,
\begin{equation}
\mathcal J=\int_{(0,\infty)^n}
\left.\partial_s^{n-3}\left[
\widetilde u'(s)e^{\widetilde u(s)\sum_i x_i}\right]\right|_{s=0}
\prod_{i=1}^nx_i^{-1/2}K_{i\lambda_i}(x_i)\,dx_i.
\end{equation}
For $t\le-1$, the required Bessel integral is
\begin{align}\label{eq:corrected-bessel-legendre}
\int_0^\infty x^{-1/2}K_{i\lambda}(x)e^{tx}\,dx
&=\frac{(2\pi)^{3/2}}{4\cosh(\pi\lambda)}
{}_2F_1\!\left(\frac14+\frac{i\lambda}{2},
\frac14-\frac{i\lambda}{2};1;1-t^2\right)\nonumber\\
&=\frac{(2\pi)^{3/2}}{4\cosh(\pi\lambda)}
P_{-1/2-i\lambda}(-t).
\end{align}
{Both sides are analytic in $t$ in a neighborhood of $-1$, so the identity extends to that neighborhood.}
It follows, initially in the common convergence strip, that
\begin{align}
\mathcal J={}&\left(\frac{(2\pi)^{3/2}}4\right)^n
\prod_{i=1}^n\frac1{\cosh(\pi\lambda_i)}\nonumber\\
&\times\left.\partial_s^{n-3}\left[
\widetilde u'(s)
\prod_{i=1}^nP_{-1/2-i\lambda_i}(-\widetilde u(s))
\right]\right|_{s=0}.
\end{align}
By Proposition~\ref{same of u} and Theorem~\ref{thm:string equ},
\[
\widetilde u(s)=-\left(1+A_\gamma
u_{\rm st}\left(-\frac{s}{L_\mu(\gamma)}\right)\right).
\]
Writing $m=n-3$ and applying the chain rule gives
\begin{align}
&\left.\partial_s^m\left[
\widetilde u'(s)
\prod_iP_{-1/2-i\lambda_i}(-\widetilde u(s))
\right]\right|_{s=0}\nonumber\\
&\qquad=(-1)^{n-2}\frac{\gamma^4\pi^2}{4}
L_\mu(\gamma)^{-(n-2)}V_{0,n}^{\rm WP}(\boldsymbol\lambda).
\end{align}
Substitution proves~\eqref{cwp} and~\eqref{eq:corrected-gamma-wp}.
The meromorphic continuation follows because only a finite jet at
$s=0$ is used.
\end{proof}
}
{
\begin{corollary}
Under the algebraic continuation specified above, for $n\ge3$, after algebraically continuing the explicit expression
\eqref{eq:corrected-gamma-wp} from the LQG range $\gamma^2\in(8/3,4)$
to a punctured neighborhood of $\gamma=0$,
\begin{equation}
V_{0,n}^{\rm WP}\!\left(\frac{2b_1}{\pi\gamma^2},\ldots,
\frac{2b_n}{\pi\gamma^2}\right)
=\sum_{s=0}^{n-3}\ \sum_{\substack{\rho\vdash s\\\ell(\rho)\le n}}
C_{\rho,n}(\gamma)\,m_\rho(b_1^2,\ldots,b_n^2),
\end{equation}
where $m_\rho$ is the monomial symmetric polynomial indexed by the
partition $\rho$, and $C_{\rho,n}(\gamma)\in\mathbb Q[\pi^2,\gamma^2]$.
\end{corollary}

\begin{proof}
Near $t=0$,
\begin{equation}
P_{-1/2-i\lambda}(1+t)
=\sum_{r=0}^\infty
\frac{(\frac12+i\lambda)_r(\frac12-i\lambda)_r}{(r!)^2}
\left(-\frac t2\right)^r.
\end{equation}
The coefficient of $t^r$ is a polynomial of degree $r$ in $\lambda^2$.
Only terms of total $t$-degree at most $n-3$ survive the derivative in
\eqref{eq:corrected-gamma-wp}.  Lagrange inversion applied to
\eqref{simplehyp}, followed by
$\lambda_i=2b_i/(\pi\gamma^2)$, gives the stated finite symmetric
polynomial.  A coefficient indexed only by the total degree would be
insufficient, since different partitions of the same degree generally
have different coefficients.
\end{proof}

For example, for the unscaled spectral variables $b_i$,
\begin{align*}
V_{0,4}^{\rm WP}(b_1,b_2,b_3,b_4)
&=2\pi^2+\frac{6\pi^2}{(8/\gamma^2)^2}
+\frac12\sum_{i=1}^4\left(\frac{\pi\gamma^2}{2}b_i\right)^2,
\end{align*}
and
\begin{align*}
V_{0,5}^{\rm WP}(b_1,\ldots,b_5)
&=10\pi^4+\frac{56\pi^4}{(8/\gamma^2)^2}
+\frac{104\pi^4}{(8/\gamma^2)^4}\\
&\quad+\left(3\pi^2+\frac{10\pi^2}{(8/\gamma^2)^2}\right)
\sum_i\left(\frac{\pi\gamma^2}{2}b_i\right)^2\\
&\quad+\frac12\sum_{i<j}
\left(\frac{\pi\gamma^2}{2}b_i\right)^2
\left(\frac{\pi\gamma^2}{2}b_j\right)^2
+\frac18\sum_i\left(\frac{\pi\gamma^2}{2}b_i\right)^4.
\end{align*}

The Weil--Petersson volume of $\mathcal M_{g,n}$ is
\[
\operatorname{Vol}_{\rm WP}(\mathcal M_{g,n})
=\frac1{(3g-3+n)!}\int_{\overline{\mathcal M}_{g,n}}
\omega_{\rm WP}^{3g-3+n}.
\]
For geodesic boundary lengths $L_1,\ldots,L_n$, symplectic reduction
gives \cite{mirzakhani2015towards}
\[
\operatorname{Vol}_{\rm WP}(\mathcal M_{g,n}(L_1,\ldots,L_n))
=\frac1{(3g-3+n)!}\int_{\overline{\mathcal M}_{g,n}}
\left(\omega_{\rm WP}+\sum_{i=1}^n\frac{L_i^2}{2}\psi_i\right)^{3g-3+n}.
\]

\begin{theorem}\label{gamma 0 wp}
Under the algebraic continuation specified above, for $n\ge3$,
\begin{equation}
\lim_{\gamma\to0}V_{0,n}^{\rm WP}\!\left(
\frac{2b_1}{\pi\gamma^2},\ldots,
\frac{2b_n}{\pi\gamma^2}\right)
=\operatorname{Vol}_{\rm WP}(\mathcal M_{0,n}(b_1,\ldots,b_n)).
\end{equation}
The convergence is coefficientwise as a polynomial in
$b_1^2,\ldots,b_n^2$, and hence locally uniform in $\boldsymbol b$.
\end{theorem}

\begin{proof}
Write $U_\gamma=u_{\rm st}$.  For every fixed order, the coefficients in
\eqref{simplehyp} converge as $\gamma\to0$ to those of
\begin{equation}\label{eq:classical-string-limit}
-x=\sum_{j=1}^\infty
\frac12\frac{\pi^{2j-2}}{j!(j-1)!}U_0(x)^j.
\end{equation}
The analytic implicit-function theorem therefore gives convergence of
every fixed jet of $U_\gamma$ at zero.  Moreover, for each fixed $r$,
\[
\frac{(\frac12+i\frac{2b}{\pi\gamma^2})_r
(\frac12-i\frac{2b}{\pi\gamma^2})_r}{(r!)^2}
\left(-\frac{A_\gamma}{2}\right)^r
\longrightarrow\frac{(-b^2/4)^r}{(r!)^2}.
\]
Because only a finite jet occurs, we pass termwise to the limit in
\eqref{eq:corrected-gamma-wp}; the result is
\[
-\frac12\left.\partial_x^{n-3}\left[
U_0'(x)\prod_{i=1}^n\mathcal J_0(b_i\sqrt{U_0(x)})
\right]\right|_{x=0},
\]
where $\mathcal J_0(z)=\sum_{r\ge0}(-1)^r(z/2)^{2r}/(r!)^2$.
For the Weil--Petersson specialization in Proposition~\ref{gwp}, the
string solution satisfies $u_0(x)=-U_0(x)/2$.  Hence
$\mathcal I_0(b_i\sqrt{2u_0})=\mathcal J_0(b_i\sqrt{U_0})$ and
$u_0'=-U_0'/2$.  Proposition~\ref{gwp} identifies the last display with
the claimed Weil--Petersson volume.
\end{proof}
}

\appendix
\section{Topological gravity and Weil--Petersson volume}

{This appendix gives a closed formula for the
Weil--Petersson volume of a sphere with $n$ boundary components.}
{Let $\Sigma_{g,n}$ be a closed genus-$g$ Riemann surface
with marked points $p_1,\ldots,p_n$, and let $\mathcal M_{g,n}$ be its
moduli space.  For $n>0$, the tautological line bundle $\mathcal L_i$ on
$\overline{\mathcal M}_{g,n}$ has fiber $T_{p_i}^*C$ over
$(C,p_1,\ldots,p_n)$, and
$\psi_i=c_1(\mathcal L_i)\in H^2(\overline{\mathcal M}_{g,n};\mathbb Q)$.
For the forgetful map
$p:\overline{\mathcal M}_{g,n+1}\to\overline{\mathcal M}_{g,n}$, set
$\kappa_1=p_*(\psi_{n+1}^2)$.}

The class $\kappa_1$ is proportional to the Weil--Petersson K\"ahler form $\omega_\text{WP}$ 
\begin{equation} \label{eq:WPKahler}
  [\omega_\text{WP}] = 2 \pi^2 \kappa_1 \ .
\end{equation}
{Their intersection numbers are denoted by}
\begin{equation} \label{topologicalgravitycorrelationfunctions}
  \left\langle \kappa_1^\ell\tau_{d_1} \ldots \tau_{d_n} \right\rangle_{g,n} = \int_{\overline{\mathcal{M}}_{g,n}}\kappa_1^\ell \psi_1^{d_1} \ldots \psi_{n}^{d_n} \ ,
  \quad \ell,d_1,\ldots,d_n \in \mathbb{Z}_{\ge 0} \ ,
\end{equation}  
{where $\tau_{d_i}$ is Witten's notation for the insertion
$\psi_i^{d_i}$.  The correlator can be nonzero only if
$\ell+d_1+\cdots+d_n=3g-3+n$.  All moduli spaces and correlators are
taken in the stable range $2g-2+n>0$; unstable correlators appearing
formally in the generating series below are set to zero.}

{Introduce the generating functions}
\begin{equation} \label{generatingfunction}
  F(\{t_k\}) = \sum_{g=0}^{+\infty} g_s^{2g} \left\langle e^{\sum_{d=0}^{\infty} t_d \tau_d} \right\rangle_g 
   =\sum_{g=0}^{+\infty} g_s^{2g} \sum_{\{n_d\}} \left(\prod_{d=0}^{\infty}\frac{t_d^{n_d}}{n_d !} \right) \left\langle \tau_0^{n_0} \tau_1^{n_1}  \ldots \right\rangle_g \ ,
\end{equation}  
and 
\begin{multline}\label{eq:G}
  G(s,\{t_k\}) = \sum_{g=0}^{+\infty} g_s^{2g} \left\langle e^{s \kappa_1 + \sum_{d=0}^{\infty} t_d \tau_d} \right\rangle_g 
   =\sum_{g=0}^{+\infty} \sum_{m=0}^{+\infty}\frac{g_s^{2g} s^m}{m!}\sum_{\{n_d\}} \left(\prod_{d=0}^{\infty} \frac{t_d^{n_d}}{n_d !} \right)
   \left\langle \kappa_1^m \tau_0^{n_0} \tau_1^{n_1}  \ldots \right\rangle_g \ ,
\end{multline}  
{where $g_s$ is the genus-expansion parameter.  The two
generating functions are related as follows.}
\begin{proposition}[{\cite[formula 2.7]{mirzakhani2015towards}}]
    \begin{equation}\label{eq:gammak}
  G(s,\{t_k\}) = F(\{t_k + \gamma_k s^{k-1} \}) \ , \qquad \gamma_0 = \gamma_1 = 0 \ , \quad \gamma_k = \frac{(-1)^k}{(k-1)!}\textcolor{red}{,\quad} k\geq 2
\end{equation}
\end{proposition}
{Denote by $F_g$ and $G_g$ the coefficients of $g_s^{2g}$.}
{The genus-zero partition function
$F_0(t_0,t_1,\ldots)$ has the following closed form
\cite[Lemma 4]{itzykson1992combinatorics}.  Introduce}
\begin{align}
    I_n\left(v,\left\{t_k\right\}\right):=\sum_{\ell=0}^{\infty} t_{n+\ell} \frac{v^{\ell}}{\ell !} \quad(n \geq 0)
\end{align}
{and set $u_0:=\partial_0^2F_0$, where
$\partial_k:=\partial_{t_k}$.}
\begin{proposition}[{\cite[Lemma 4]{itzykson1992combinatorics}}]\label{F_0}
{We have}
\begin{align}
   &F_0=\frac{1}{2}\int_0^{u_0}(I_0(v,\{t_k\})-v)^2dv\\
    &I_0(u_0,\{t_k\})=u_0\\
    &\partial_k u_0=\partial_0\frac{u_0^{k+1}}{(k+1)!}
\end{align}
\end{proposition}

Define the formal differential operator
\[
V(b)=\sum_{k=0}^\infty\frac{b^{2k}}{2^k k!}\partial_k
\]
and, for $\mathbf b=(b_1,\ldots,b_n)$,
\begin{align}
V_{g,n}^{\rm GWP}(\mathbf b,\{t_k\})
&:=V(b_1)\cdots V(b_n)F_g(\{t_k\})\nonumber\\
&=\left\langle e^{\sum_{d\ge0}t_d\tau_d}
\prod_{i=1}^ne^{b_i^2\psi_i/2}\right\rangle_{{g}}.
\end{align}
\begin{remark}
By~\eqref{eq:WPKahler} and~\eqref{eq:gammak}, the ordinary
Weil--Petersson specialization is
\[
t_0=t_1=0,\qquad
t_k=\gamma_k(2\pi^2)^{k-1}\textcolor{red}{,\quad} k\ge2.
\]
With these couplings,
$V_{g,n}^{\rm GWP}(\mathbf b,\{t_k\})=V_{g,n}^{\rm WP}(\mathbf b)$.
\end{remark}

Set $g_s=1$.  {For $\beta>0$,} introduce
\begin{align}
B(\beta)&:=\int_0^\infty
\frac{b e^{-b^2/(2\beta)}}{\sqrt{2\pi\beta}}V(b)\,db
=\frac1{\sqrt{2\pi}}\sum_{k=0}^\infty
\beta^{k+1/2}\partial_k.
\end{align}
\begin{proposition}
For $n\ge3$ {and $\beta_1,\ldots,\beta_n>0$},
\begin{align}
\prod_{i=1}^nB(\beta_i)F_0
&=\sqrt{\frac{\prod_{i=1}^n\beta_i}{(2\pi)^n}}
\frac{\partial_0^{n-2}e^{(\sum_i\beta_i)u_0}}
{\sum_i\beta_i}\nonumber\\
&=\sqrt{\frac{\prod_{i=1}^n\beta_i}{(2\pi)^n}}
\partial_0^{n-3}\left[(\partial_0u_0)e^{(\sum_i\beta_i)u_0}\right].
\end{align}
\end{proposition}
\begin{proof}
Expand each $B(\beta_i)$ and repeatedly use
$\partial_k u_0=\partial_0(u_0^{k+1}/(k+1)!)$ from
Proposition~\ref{F_0}.  Induction on the number of insertions gives the
first equality.  The second follows from
$\partial_0e^{(\sum_i\beta_i)u_0}
=(\sum_i\beta_i)(\partial_0u_0)e^{(\sum_i\beta_i)u_0}$.
\end{proof}

\begin{proposition}\label{gwp}
For $n\ge3$,
\begin{equation}
V_{0,n}^{\rm GWP}(\mathbf b,\{t_k\})
=\partial_0^{n-3}\left[
(\partial_0u_0)\prod_{i=1}^n\mathcal I_0(b_i\sqrt{2u_0})\right],
\end{equation}
where
$\mathcal I_0(x)=\sum_{r=0}^\infty(x/2)^{2r}/(r!)^2$.
\end{proposition}

\begin{proof}
The elementary transform identity
\[
\int_0^\infty
\frac{b e^{-b^2/(2\beta)}}{\sqrt{2\pi\beta}}
\mathcal I_0(b\sqrt{2u})\,db
=\sqrt{\frac\beta{2\pi}}e^{\beta u}
\]
shows that the $n$ Gaussian transforms of the proposed right-hand side
equal
\[
\sqrt{\frac{\prod_i\beta_i}{(2\pi)^n}}
\partial_0^{n-3}\left[(\partial_0u_0)e^{(\sum_i\beta_i)u_0}\right].
\]
The preceding proposition gives the same transforms for
$V_{0,n}^{\rm GWP}$; coefficientwise injectivity of the transform on the
formal power series in $b_1^2,\ldots,b_n^2$ proves the claim.
\end{proof}

For the Weil--Petersson specialization, set
\[
t_0=x,\qquad t_1=0,\qquad
t_k=\gamma_k(2\pi^2)^{k-1}\quad(k\ge2).
\]
Then
\begin{equation}\label{eq:wp-string}
\sum_{\ell=1}^\infty
\frac{(-1)^\ell(2\pi^2)^{\ell-1}u_0(x)^\ell}
{\ell!(\ell-1)!}=-x,
\end{equation}
and
\begin{equation}\label{eq:wp-volume-bessel}
V_{0,n}^{\rm WP}(\mathbf b)
=\left.\partial_x^{n-3}\left[
u_0'(x)\prod_{i=1}^n\mathcal I_0(b_i\sqrt{2u_0(x)})
\right]\right|_{x=0}.
\end{equation}

\section{{Proof of the integral formula~\eqref{bessel int formula}}}

{We use the following two quoted integral formulas.}
\begin{lemma}[{\cite[Volume 2, Section 7.3. (15)]{Bateman1953HigherTF}}]
For $\Re(\nu)>-\frac{1}{2}$, $\Re(z)>0$,
    $$ K_\nu(z)=\frac{\pi^{1 / 2}}{\Gamma(\nu+\frac{1}{2})} (\frac{z}{2})^\nu \int_1^{\infty} e^{-z t}\left(t^2-1\right)^{\nu-1 / 2} d t $$
\end{lemma}

\begin{lemma}[{\cite[Volume 1, Section 2.12. (5)]{Bateman1953HigherTF}}]
For $\Re(c)>\Re(b)>0$, $|{\arg}(z)|<\pi$,
\begin{equation}
{}_{2}F_{1}(a, b ; c ; 1-z)=\frac{\Gamma(c)}{\Gamma(b) \Gamma(c-b)} \int_0^{\infty} s^{b-1}(1+s)^{a-c}(1+s z)^{-a} d s
\end{equation}
\end{lemma}

\begin{proof}[{Proof of Prop. \ref{besselint}}]
\begin{align*}
    \int_0^{\infty} x^{\mu-1} e^{-\alpha x} K_\nu(\beta x) d x&=\frac{\pi^{1/2}}{\Gamma(\nu+\frac{1}{2})}(\frac{\beta}{2})^\nu\int_{1}^{\infty}\Big(\int_{0}^\infty x^{\mu+\nu-1}e^{-(\alpha+\beta t)x}dx\Big)(t^2-1)^{\nu-\frac{1}{2}}dt\\
    &=\frac{\pi^{1/2}}{\Gamma(\nu+\frac{1}{2})}(\frac{\beta}{2})^\nu\Gamma(\mu+\nu)\int_{1}^\infty\frac{(t^2-1)^{\nu-\frac{1}{2}}}{(\alpha+\beta t)^{\mu+\nu}}dt\\
    &=\frac{\pi^{1/2}}{\Gamma(\nu+\frac{1}{2})}(\frac{\beta}{2})^\nu\frac{\Gamma(\mu+\nu)}{(\alpha+\beta)^{\mu+\nu}}\int_{1}^\infty\frac{(t^2-1)^{\nu-\frac{1}{2}}}{(1 +\frac{2\beta}{\alpha+\beta} \frac{t-1}{2})^{\mu+\nu}}dt\\
     &{=\frac{\pi^{1/2}}{\Gamma(\nu+\frac{1}{2})}(\frac{\beta}{2})^\nu\frac{\Gamma(\mu+\nu)}{(\alpha+\beta)^{\mu+\nu}}\int_{0}^\infty
     [4s(1+s)]^{\nu-\frac12}
     \left(1+\frac{2\beta}{\alpha+\beta}s\right)^{-\mu-\nu}2\,ds}\\
     &{=\frac{\pi^{1/2}}{\Gamma(\nu+\frac{1}{2})}(2\beta)^\nu\frac{\Gamma(\mu+\nu)}{(\alpha+\beta)^{\mu+\nu}}\int_{0}^\infty s^{\nu-\frac{1}{2}}(1+s)^{\nu-\frac{1}{2}}\left(1+\frac{2\beta}{\alpha+\beta}s\right)^{-\mu-\nu}ds}\\
     &=\frac{\sqrt{\pi}(2 \beta)^\nu}{(\alpha+\beta)^{\mu+\nu}} \frac{\Gamma(\mu+\nu) \Gamma(\mu-\nu)}{\Gamma\left(\mu+\frac{1}{2}\right)} {}_{2}F_{1}\left(\mu+\nu, \nu+\frac{1}{2} ; \mu+\frac{1}{2} ; \frac{\alpha-\beta}{\alpha+\beta}\right)
\end{align*}
The calculation is initially justified when
$\Re\nu>-\frac12$, $\Re(\mu\pm\nu)>0$, $\beta>0$, and
$\Re(\alpha+\beta)>0$.  The identity on the full domain stated in
Proposition~\ref{besselint} follows from $K_\nu=K_{-\nu}$ and analytic
continuation in the parameters, with all complex powers on their principal
branches.
\end{proof}

\bibliographystyle{alpha}
\bibliography{cibib}

\end{document}